\documentclass[12pt,twoside,english,reqno,a4paper,final]{amsart}
\usepackage[utf8]{inputenc}
\usepackage[english]{babel}
\usepackage{cancel}
\usepackage[lined,boxed]{algorithm2e}
\usepackage{float} 

\usepackage{hyperref}
\usepackage{fullpage}

\usepackage{listings,graphicx,amsmath,varioref,amscd,amssymb,color,xcolor,bm,amsthm,amsfonts,graphics,lineno,float,comment}
\usepackage{mathrsfs} 

\theoremstyle{plain}
\newtheorem{theorem}{Theorem}[section]

\newtheorem{lemma}[theorem]{Lemma}

\newtheorem{Corollary}[theorem]{Corollary}

\theoremstyle{definition}
\newtheorem{defin}[theorem]{Definition}

\newtheorem{remark}[theorem]{Remark}

\theoremstyle{remark}

\numberwithin{equation}{section}
\def\supp{\text{\text{supp}}}

\newcommand{\car}[1]{\raise1pt\hbox{$\chi$}_{#1}}

\definecolor{sap}{RGB}{120,36,51}
\def\R{\mathbb{R}}
\def\N{\mathbb{N}}

\def\RN{\mathbb{R}^{N}}

\def\sob{W^{1,p}_{0}(\Omega)}
\def\linf{L^{\infty}(\Omega)}

\def\lp'n{(L^{p'}(\Omega))^{N}}

\def\car#1{\chi_{_{{#1}}}}

\def\norma#1#2{\|#1\|_{\lower 4pt \hbox{$ \scriptstyle #2$ }}}

\author[R. Durastanti]{Riccardo Durastanti}

\author[F. Oliva]{Francescantonio Oliva}

\address[R. Durastanti]{Dipartimento di Matematica e Applicazioni ``R. Caccioppoli'', Universit\`a di Napoli Federico II, Via Cintia, Monte S. Angelo, 80126 Napoli, Italy
	\\ riccardo.durastanti@unina.it}

\address[F. Oliva]{Dipartimento di Scienze di Base e Applicate per l'Ingegneria, Sapienza Università di Roma, Via Antonio Scarpa 16, 00161 Roma, Italy
\\ francescantonio.oliva@uniroma1.it}


\begin{document}

\title{The Dirichlet problem for possibly singular elliptic equations with degenerate coercivity}

\begin{abstract}
We deal with existence, uniqueness and regularity of nonnegative solutions to a Dirichlet problem for equations as
	\begin{equation*}
\displaystyle -\operatorname{div}\left(\frac{|\nabla u|^{p-2}\nabla u}{(1+u)^{\theta(p-1)}}\right) = h(u)f  \quad\text{in }\Omega,
	\end{equation*}
where $\Omega$ is an open bounded subset of $\mathbb{R}^N$ ($N\ge 2$), $p>1$, $\theta\ge 0$, $f\geq 0$ belongs to a suitable Lebesgue space and $h$ is a continuous, nonnegative function which may blow up at zero and it is bounded at infinity.
\end{abstract}

\maketitle

\tableofcontents

\section{Introduction}

In this paper we study existence, uniqueness and regularity of solutions to problems as
\begin{equation}
	\begin{cases}
		\displaystyle -\operatorname{div}\left(\frac{a(x,\nabla u)}{(1+u)^{\theta(p-1)}}\right) = h(u)f & \text{in}\ \Omega, \\
		u\ge 0  & \text{in}\ \Omega,\\
		u=0 &\text{on}\ \partial\Omega,
	\end{cases}
	\label{pbintro}
\end{equation}
where $\Omega$ is an open bounded subset of $\mathbb{R}^N$ ($N\geq 2$),  $a(x,\xi)$ is a nonlinear Carath\'eodory function satisfying the standard Leray-Lions assumptions and which can be modelled by $a(x,\xi)=|\xi|^{p-2}\xi$ in the simplest case ($1<p<N$), $\theta\ge0$ and  $f$ is a nonnegative function which can be merely integrable. Finally $h(s)$ is a continuous, nonnegative function which behaves as  $s^{-\gamma_1}$ near zero and as $s^{-\gamma_2}$ at infinity with $\gamma_1,\gamma_2\ge 0$.

\medskip

At first sight problem \eqref{pbintro} presents some mathematical peculiarities which deserve to be pointed out. Firstly let observe that \eqref{pbintro} could be singular in the following sense: the solution is required to be zero on the boundary of the domain but, simultaneously, the right-hand side of \eqref{pbintro} could blow up.
\\
Another important feature is the lack of coercivity for positive $\theta$ of the operator $$\operatorname{div}\left(\frac{a(x,\nabla u)}{(1+u)^{\theta(p-1)}}\right)$$ which acts between $W^{1,p}_0(\Omega)$ and $W^{-1,p'}(\Omega)$. Just to give an idea, let us assume for a moment that $a(x,\xi)= |\xi|^{p-2}\xi$, $h(s)f\equiv 1$ and let us formally multiply the equation in \eqref{pbintro} by $u$. Then one simply gets that $|\nabla u|^p(1+u)^{-\theta(p-1)}$ is integrable which, \textit{a priori}, does not give clear informations on the solution itself for large values. \\
Hence, the above mentioned features do not allow to apply standard existence and uniqueness theorems for solutions to \eqref{pbintro}.

\medskip

The literature involving problems similar to \eqref{pbintro} is endless; here we limit ourselves to give a very brief description of some papers which mostly influenced us.

\medskip

Problem \eqref{pbintro} in the coercive case (i.e. $\theta=0$) has been extensively studied in the past. In presence of the Laplace operator and when $h(s) = s^{-\gamma}$ ($\gamma>0$) the first interesting researches  come from the pioneering works \cite{crt, lm, stuart}. Here existence and uniqueness of classical solutions is shown for smooth enough $f$'s. When $f$ is a function in a Lebesgue space (or a measure) and $h(s)=s^{-\gamma}$, we mention \cite{bo} where the existence of a distributional solution is proved along with the regularizing effect given by the fact that $h$ goes to zero at infinity; just to give an idea, if $\gamma=1$ and one formally multiply the equation in \eqref{pbintro} by $u$ ($p=2$), then the solution belongs to $H^1_0(\Omega)$ even for an $L^1$-datum $f$. Another issue is that, in case $\gamma>1$, the solution belongs only locally to $H^1(\Omega)$ and the boundary datum is given as a suitable power of the solution having zero Sobolev trace. When the operator is nonlinear with classical Leray-Lions structure and $h$ is a general not necessarily monotone function one can refer to various works \cite{ddo,do,op,op2}. In particular, among other things, it is proven the existence of a distributional solution under the assumption that $f$ is just a Radon measure.

\smallskip

On the other side uniqueness is a more subtle theme; a natural request, which we assume in the next few lines, is that $h$ is  nonincreasing. For our scope, we also limit the presentation to the case of solutions with zero Sobolev trace, which (in general) are the ones found in case  $\gamma_1\le 1$ and which is mainly the content of the present paper.  Firstly we highlight that distributional solutions with $W^{1,p}$-finite energy are always unique; this is shown in various papers and it relies on an extension argument of the set of test functions and on a classical comparison technique (see for instance \cite{boca,o,durOl}).
On the other hand, explicit examples show that the solutions may have infinite energy if the datum is merely integrable. However, any truncation of the solution has $W^{1,p}$-finite energy. In this respect in \cite{ddo}, if $\gamma_1\le 1$, working in the framework of the (so called) renormalized solutions, it is proven uniqueness of solutions. We also highlight that uniqueness of distributional solutions is proved in \cite{op2} for linear operator. Finally we quote \cite{cst} where the authors show uniqueness of solutions in presence of the $p$-Laplace operator and for a sufficiently regular function $f$.

\smallskip

For more and different aspects concerning singular problems we refer to \cite{am,bct,car,diaz,dur,DG,fs,gmm,gg,locsc,sz}.

\medskip

Let us now briefly summarize some known results in the noncoercive case (i.e. $\theta>0$) when $h\equiv 1$.
When $f$ is a measurable function with suitable integrable properties we refer to \cite{abfo} for the existence of solutions. The result is obtained by means of approximation through suitable coercive problems. In particular, the authors show that if $f\in L^1(\Omega)$ a solution exists as long as $\theta<1$. For $\theta=1$ the situation is more tricky: a bounded and finite energy solution to \eqref{pbintro} is proven to exist if $f\in L^m(\Omega)$ with $m>\frac{N}{p}$. When $\theta>1$ nonexistence of solutions is shown for datum $f$ with norm large enough. Let us also mention that in \cite{b}, the author proves existence of a solution when $\theta=1$ and $f\in L^{\frac{N}{p}}(\Omega)$. \\
A particular attention needs to be given to the notion of solution employed in these works. Indeed, for functions $f$ merely integrable and $\theta$ sufficiently near to one, then, in general, $|\nabla u|^{p-1}$ is not locally integrable. This fact, which can be confirmed by explicit examples, does not permit to consider distributional solutions to \eqref{pbintro}. As we will see below, the right framework to work in is the entropy one which heavily relies on properties of any truncation of the solution.

\smallskip

As expected uniqueness in this noncoercive case is more delicate and the literature is more limited. Again for $h\equiv 1$ we mainly refer to \cite{p} where the author
proves uniqueness of entropy solutions under stronger assumptions on $a$ and once that $f\in L^m(\Omega)$ with
\begin{equation}
\label{exppor}
m\ge \frac{N(2-\theta)}{2+N(1-\theta)}.
\end{equation}
It is worth underlining that the previous uniqueness result could be extended to the case of a continuous and bounded function $h(s)$, which is nonincreasing in $s$. Finally we also quote \cite{p2} where the author shows uniqueness of entropy solutions to \eqref{pbintro} in case $a(x,\nabla u)=a(x)\nabla u$ and $h\equiv 1$ with $f\in L^1(\Omega)$, thanks to a suitable change of variable, which, however, is not available under our assumptions.

\smallskip

For more concerning problems with degenerate coercivity we refer to \cite{abo,blgui,bg, bb, bdo,  dellapietra, dellapdibl, fgmz, gp, huang,  smarrazzo, sou, trombetti,zl}.

\medskip

In this work we deal with problems as in \eqref{pbintro} possibly in presence of both a noncoercive principal operator and a general lower order term; in particular the function $h$ may be singular and without any monotonicity property. In this case, to the best of our knowledge, there are no results in literature about existence and uniqueness of solutions. Our aim is to extend and improve both the existence and uniqueness results listed above.

\smallskip

The existence of a solution is obtained by the means of an approximation process to \eqref{pbintro}. As one can image, the result follows by unifying truncation techniques typical of noncoercive operators with methods employed in dealing with functions possibly blowing up at the origin. In this respect we will prove the existence of an entropy solution if $\gamma_1\le 1$ (see Theorem \ref{teo_ent}). \\
Firstly we are interested into the regularizing effect given by the behaviour of $h$ at infinity (i.e. $\gamma_2$).
Indeed we obtain improvements over the $\theta$'s range (i.e. $0\leq \theta\leq 1$) for which existence of solutions is known: in particular an entropy solution exists for any $0\leq \theta\leq 1+\frac{\gamma_2}{p-1}$. An interesting fact to be pointed out for future developments is that, once $\gamma_2>0$, the threshold depends both on $\gamma_2$ and $p$ and it blows up as $p\to 1^+$. Moreover, another regularizing effect we get is that, if $h$ touches $0$ in some point, then an entropy bounded solution exists without any restriction on $\theta$ (see Section \ref{s8}). \\
As already mentioned, the second interesting aspect is providing existence for a very general function $h$, which possibly blows up at the origin. On this matter, one needs to take care of the zone where the solution is small. It is also worth mentioning that existence of solutions to \eqref{pbintro} in the coercive case (i.e. $\theta=0$) is often recovered by an approximation argument jointly to  an application of the maximum principle which assures that the solutions are far away from zero. When $\theta>0$ one needs to treat the zone where the solution is small in a different way and through the help of suitable test functions avoiding the use of the maximum principle which, in this generality, is not always applicable.

\medskip

Besides existence, we investigate the regularity of entropy solutions in dependence of the summability of the datum $f$. The discussion on regularity is not only interesting in itself, but, since the properties proven are for any entropy solution and not just for those obtained by approximation, this is fundamental for the proof of the uniqueness result.\\
The technique behind the uniqueness strongly relies on using the right-hand side of the equation in \eqref{pbintro}. In particular, if $\theta\leq 1+\frac{\gamma_2}{p-1}$, there is at most an entropy solution of \eqref{pbintro} under the assumptions that $h$ is a decreasing function and $f\in L^m(\Omega)$ is positive almost everywhere in $\Omega$ with
$$
m\ge \max\left(\frac{N(p-1)}{(N-p)(\gamma_2-\theta(p-1))+N(p-1)}, 1\right)
$$
(see Theorem \ref{teo_uniqueness}). In case $p=2$ and $\gamma_1=\gamma_2=0$ (i.e. the case for which $h$ may be a continuous and bounded function in $[0,\infty)$ without decaying at infinity), the lower bound on $m$ becomes
$$
m\ge \frac{N}{N-\theta(N-2)}.
$$
This bound on $m$ is lower than the one in \eqref{exppor}; this means that, under the additional assumption that $h$ is decreasing and the positiveness of $f$, but under weaker assumptions on $a$, one gets uniqueness for less regular $f$'s. \\ Moreover we underline that this uniqueness result holds replacing $(1+u)^{-\theta(p-1)}$ in \eqref{pbintro} with a Lipschitz function $b(u)$ and does not depend on the singularity of $h$ at 0 (i.e. $\gamma_1$).

\medskip

Finally we exploit the connection between the entropy solution and the distributional one, proving that, whenever $f$ satisfies a certain summability request, an entropy solution is also a distributional one (see Lemma \ref{entdis}). At the end of the paper we give a radial example which confirms the optimality of this request on $f$. Furthermore we briefly comment the existence of a distributional solution for the strongly singular case (i.e. $\gamma_1>1$).

\medskip

The plan of the paper is as follows. In Section \ref{s2} we fix the notations and the mathematical preliminaries used throughout the work. In Section \ref{s3} we precise the mathematical problem and we state the main results. In Section \ref{s4} we prove existence of an entropy solution. In Section \ref{sec_reg} we show the regularity of any entropy solution. In Section \ref{s5} we deal with the uniqueness. In Section \ref{s7} we study the connection between entropy, renormalized and distributional solutions and in Section \ref{s8} we conclude with some remarks and examples.

\section{Notations and preliminaries}
\label{s2}

In the entire paper $\Omega$ is an open and bounded subset of $\RN$, with $N\geq 2$. We denote by $\partial A$ the boundary and by $|A|$ the Lebesgue measure of a measurable subset $A$ of $\RN$. By $C^1_c(\Omega)$ we mean the space of $C^1$ functions with compact support in $\Omega$.

\smallskip

For any $1\leq p<N$, $p^*=\frac{Np}{N-p}$ is the Sobolev  conjugate exponent of $p$, and $\mathcal{S}$ is the best constant in the Sobolev inequality for functions in $W^{1,p}_0(\Omega)$.

\smallskip

We denote by $\chi_E$ the characteristic function of $E\subset\Omega$, namely
$$
\chi_E(x)=
\begin{cases}
1 & x\in E,\\
0 & x\in \Omega\setminus E,
\end{cases}
$$
and by $f^+:=\max(f,0), f^-:= -\min(f,0)$ the positive and the negative part of a function $f$.

\smallskip

For a fixed $k>0$ and $s\in\R$, we introduce the following truncation functions
\begin{equation}\label{defTk}
T_k(s)=\max (-k,\min (s,k)),
\end{equation}
\begin{equation}\label{defGk}
G_k(s)=(|s|-k)^+ \operatorname{sign}(s),
\end{equation}
and
\begin{equation}
\label{Vdelta}
\displaystyle
V_{k}(s)=
\begin{cases}
1 \ \ &|s|\le k, \\
\displaystyle\frac{2k-s}{k} \ \ &k <|s|< 2k, \\
0 \ \ &|s|\ge 2k,
\end{cases}
\end{equation}
which will be widely used in the sequel both with $k \to 0^+$ and $k\to \infty$.

\smallskip

For the sake of completeness, we recall two well-known inequalities which will be useful in what follows:
\begin{itemize}
\item[$(i)$] For any $\underline x>0$ and $t>0$ there exists a positive constant $\underline{c}$ dependent only on $\underline x, t$ such that
\begin{equation}
\label{as1}
(1+x)^t-1 \leq \underline c x \qquad\forall x\in [0,\underline x];
\end{equation}
\item[$(ii)$] For any $\overline x>0$, $a\geq 0$, $b\geq 0$ and $t>0$ there exists a positive constant $\overline c$ dependent only on $\overline x, a, b, t$ such that
\begin{equation}
\label{as3}
(a+bx)^t \leq \overline c x^t \qquad\forall x\in [\overline x,\infty).
\end{equation}
\end{itemize}

\medskip

Since we will deal with functions $u$ which not necessarily belong to  $W^{1,1}_{\rm loc}(\Omega)$, we clarify the meaning of $\nabla u$.
\begin{lemma}[Lemma $2.1$ of \cite{BBGGPV}]
\label{dalmaso2}
Let $u:\Omega\to\R$ be a measurable function such that $T_k(u)\in W^{1,1}_{\rm loc}(\Omega)$ for every $k>0$. Then there exists a unique measurable function $v:\Omega\to\R^N$ such that
$$
\nabla T_k(u)=v\chi_{\{|u|\leq k\}} \quad\text{for a.e. }x\in\Omega \text{ and for every } k>0.
$$
Furthermore, $u\in W^{1,1}_{\rm loc}(\Omega)$ if and only if $v\in L^1_{\rm loc}(\Omega)$, and then $v= \nabla u$ in the usual distributional sense.
\end{lemma}

From now onwards, when dealing to $\nabla u$ of a function $u$ such that $T_k(u)\in W^{1,1}_{\rm loc}(\Omega)$ for every $k>0$, we refer to the function $v$ given by Lemma \ref{dalmaso2}. We note that $v$ may not be a locally integrable function, and so, in this case, $v$ is not the derivative in distributional sense of $u$.

\smallskip

We recall the definition of Marcinkiewicz space $M^s(\Omega)$: a function $u: \Omega \to \R$ belongs to $M^s(\Omega)$, with $s>0$, if there exists a positive constant $c$ such that
\begin{equation}
\label{marc-eq}
|\{x\in \Omega : |u(x)| > \lambda\}|\leq \frac{c}{\lambda^s} \quad \forall \lambda >0,
\end{equation}
and, then, we define
$$
\displaystyle \|u\|_{M^s(\Omega)}:=\left(\inf\{ c>0 : \eqref{marc-eq} \text{ holds}\}\right)^{\frac{1}{s}}.
$$
We recall that $L^1(\Omega)\subset M^1(\Omega)$ and $L^{s}(\Omega)\subset M^{s}(\Omega) \subset L^{s-\varepsilon}(\Omega)$ for every $s>1$ and $0<\varepsilon \leq s-1$.

\smallskip

If no otherwise specified, we will denote by $c,C$ several constants whose value may change from line to line. These values will only depend on the data (for instance $c,C$ may depend on $\Omega$, $N$ and $p$) but they will never depend on the indexes of the sequences we will often introduce.

\section{Basic assumptions and main results}
\label{s3}

Let $\Omega$ be a bounded open set of $\mathbb{R}^N$ ($N\ge 2$) and let us deal with the following problem
\begin{equation}
\label{pbmain}
\begin{cases}
\displaystyle -\operatorname{div}(a(x,u,\nabla u)) = h(u)f & \text{ in }\Omega,\\
u \ge 0 & \text{ in } \Omega,\\
u=0 & \text{ on } \partial\Omega,
\end{cases}
\end{equation}
where, for $1<p<N$, $\displaystyle{a(x,s,\xi):\Omega\times\mathbb{R}\times\mathbb{R}^{N} \to \mathbb{R}^{N}}$ is a Carath\'eodory function satisfying:
\begin{align}
&a(x,s,\xi)\cdot\xi\ge b(s)|\xi|^{p},
\label{cara1}\\
&|a(x,s,\xi)|\le \beta(\ell(x) + |s|^{p-1}+|\xi|^{p-1}) \quad \text{for some } \beta>0 \text{ and } 0\le \ell\in L^{\frac{p}{p-1}}(\Omega),
\label{cara2}\\
&(a(x,s,\xi) - a(x,s,\xi^{'} )) \cdot (\xi -\xi^{'}) > 0,
\label{cara3}	
\end{align}
for almost every $x$ in $\Omega$, for every $s$ in $\R$ and for every $\xi\neq\xi^{'}$ in $\mathbb{R}^N$ and where $b:\mathbb{R}\to \mathbb{R}$ is a continuous and bounded function such that
\begin{equation}\label{b}
b(s) \geq \frac{\alpha}{(1+|s|)^{\theta(p-1)}}  \quad \text{for } \alpha>0 \text{ and } \theta\ge 0.
\end{equation}
Here $f$ is nonnegative and it belongs to $L^m(\Omega)$ with $m\ge 1$. The function $h:[0,\infty)\to [0,\infty]$ is continuous and finite outside the origin such that
\begin{equation}
\exists\;\gamma_1 \ge 0, c_1,s_1>0: \  h(s) \le \frac{c_1}{s^{\gamma_1}} \quad \text{if }s\leq s_1,
\label{h1}
\end{equation} 		
and
\begin{equation}\label{h2}
\exists\;\gamma_2\ge 0, c_2,s_2>0: \ h(s) \le \frac{c_2}{s^{\gamma_2}} \quad \text{if } s\geq s_2,
\end{equation}
where $0<s_1<s_2$. Note that from \eqref{h1} and \eqref{h2} and for every $\delta>0$ one always has that $h\in L^{\infty}([\delta,\infty))$ and also that the case of a continuous and everywhere finite function $h$ is well covered by our assumptions.

Let finally observe that if $h(0)=0$ the null function is a trivial solution. Hence in the sequel we may assume $h(0)\not=0$.

\medskip

Firstly we spend a few words concerning the notion of solution to \eqref{pbmain} we adopt. As we will see, our solutions will not always be in the classical sense of the \textit{distributions}. Indeed, this framework may not be the right one if $\theta$ is large enough: the main issue is that, as the datum $f$ is rough, we will not be able to deal with solutions satisfying $a(x,u,\nabla u)\in L^1_{\rm loc}(\Omega)$ (see the radial example in Section \ref{s8}). In particular, as we will see in Section \ref{sec_reg} below, in the limit case $\theta=1+\frac{\gamma_2}{p-1}$ we will not be able to show that the solutions belong to any Marcinkiewicz space. This is consistent with the case $\theta<1+\frac{\gamma_2}{p-1}$, where the solutions always belong to a Marcinkiewicz space whose index degenerates as $\theta \to 1+\frac{\gamma_2}{p-1}$ (cf. Lemma \ref{lem_sum_m=1}). \\
Anyhow, the solution we find is always a measurable and almost everywhere finite function such that any of its truncations has finite energy, which means that the solution needs to be intended in the \textit{entropy} sense. In this setting
we prove existence and, under some additional assumptions on both the operator and the function $h$, uniqueness of solutions to \eqref{pbmain} .

\smallskip

We start giving what we precisely mean by an entropy solution to \eqref{pbmain}.

\begin{defin}
\label{defent}
A nonnegative measurable function $u$, which is almost everywhere finite in $\Omega$, is an \textit{entropy solution} to problem \eqref{pbmain} if $T_k(u) \in W^{1,p}_0(\Omega)$ for every $k>0$ and if
\begin{equation}
\label{ent0}
a(x,T_k(u),\nabla T_k(u)) \in L^{\frac{p}{p-1}}(\Omega)^N,\quad h(u)fT_k(u-\varphi) \in L^1(\Omega),
\end{equation}
and
\begin{equation}
\label{ent}
\int_{\Omega} a(x,u,\nabla u)\cdot\nabla T_k(u-\varphi) \le \int_{\Omega}h(u)fT_k(u-\varphi)
\end{equation}
for every $k>0$ and for any $\varphi \in W^{1,p}_0(\Omega)\cap L^\infty(\Omega)$.
\end{defin}

\begin{remark}
The left-hand side of \eqref{ent} is well-defined and finite. Indeed, $\nabla T_k(u-\varphi)$ is not null only on the set $\{|u-\varphi|<k\}$, and here one has $|u|< \|\varphi \|_{L^\infty(\Omega)} + k=: M$. Hence, since $T_M(u)\in \sob$, we deduce $a(x,T_M(u),\nabla T_M(u))\in L^{\frac{p}{p-1}}(\Omega)^N$ and $\nabla T_k(T_M(u)-\varphi)\in L^p(\Omega)^N$, and so
$$
\int_{\Omega} a(x,u,\nabla u)\cdot\nabla T_k(u-\varphi) = \int_{\Omega} a(x,T_M(u),\nabla T_M(u))\cdot\nabla T_k(T_M(u)-\varphi)
$$
is finite. Moreover, thanks to the second request in \eqref{ent0}, the right-hand side of \eqref{ent} is also finite. This request is natural for singular problems; indeed there are cases in which $h(u)f\not \in L^1(\Omega)$ as shown in Example 2 of \cite{op2}.
\end{remark}

Note also that we give the notion of entropy solution with the inequality sign in \eqref{ent} just for historical reasons (see \cite{BBGGPV}). Indeed, taking $\varphi=T_m(u)+T_k(u-\psi)$ in \eqref{ent} with $\psi\in W^{1,p}_0(\Omega)\cap L^{\infty}(\Omega)$ and $m>k+\|\psi\|_{L^{\infty}(\Omega)}$, using that $T_k(u-T_m(u)-T_k(u-\psi))=T_{2k}(u-T_m(u))- T_k(u-\psi)$ and letting $m$ tend to infinity it is easy to obtain the reverse inequality in \eqref{ent}. It follows that \eqref{ent} is equivalent to
\begin{equation*}
\int_{\Omega} a(x,u,\nabla u)\cdot\nabla T_k(u-\varphi) = \int_{\Omega}h(u)fT_k(u-\varphi).
\end{equation*}
See also Introduction of \cite{p2} for further details. 

\medskip

It is worth mentioning that a notion of solution which is widely employed to deal with problems as \eqref{pbmain} is the \emph{renormalized solution}. This framework is also suitable to prove uniqueness of solutions even in the noncoercive case (see for instance \cite{blgui}). We prefer not to burden the discussion by introducing the renormalization formulation here. However, in Section \ref{s7}, we precisely set the notion of renormalized solution showing that, under our assumptions, it is equivalent to the entropy one. It follows that our existence, uniqueness and regularity results hold for both entropy and renormalized solutions.

\medskip

We are now ready to state the existence of an entropy solution to \eqref{pbmain} if $\gamma_1\leq 1$ and
\begin{equation}
\label{condent}
0\leq\theta \le 1+\frac{\gamma_2}{p-1}.
\end{equation}

\begin{theorem}
\label{teo_ent}
Let $a$ satisfy \eqref{cara1}, \eqref{cara2}, \eqref{cara3}, \eqref{b} and \eqref{condent}. Let $h$ satisfy \eqref{h1} with $\gamma_1\le 1$, \eqref{h2} and let $f\in L^1(\Omega)$ be nonnegative. Then there exists an \textit{entropy solution} to \eqref{pbmain}.	
\end{theorem}

\begin{remark}
Let us just remark that $\gamma_1\le 1$ is required to find, in general, solutions with null trace in the usual Sobolev sense. This is a classical theme and for more regarding the difference between $\gamma_1\le 1$ and $\gamma_1>1$ in singular problems we quote \cite{bo, ddo, op}. We also refer to Section \ref{s8} for a brief comment concerning the case $\gamma_1>1$. \\
As concerns the threshold on $\theta$ we refer once again to Section \ref{s8} where we briefly treat the case $\theta>1+\frac{\gamma_2}{p-1}$. Let us observe that the appearing threshold is quite natural; we quote \cite{abfo} for a discussion on nonnexistence results when $h\equiv 1$ (i.e. $\gamma_1=\gamma_2=0$) and $\theta>1$.
\end{remark}

In order to prove the uniqueness result we need the following additional assumptions on the principal operator. We assume that $a(x,s,\xi) =\tilde{a}(x,\xi)b(s)$ where $b:\R\to \R$ is a bounded and Lipschitz function satisfying \eqref{b} and where $\displaystyle{\tilde{a}(x,\xi):\Omega\times\mathbb{R}^{N} \to \mathbb{R}^{N}}$ is a Carath\'eodory function satisfying:
\begin{equation}
	\begin{aligned}
\label{caraunique}
&\tilde{a}(x,\xi)\cdot\xi\ge |\xi|^{p},
\\
&|\tilde{a}(x,\xi)|\le \beta(\ell(x) + |\xi|^{p-1}) \quad \text{for some } \beta>0 \text{ and } 0\le \ell\in L^{\frac{p}{p-1}}(\Omega),
\\
&(\tilde{a}(x,\xi) - \tilde{a}(x,\xi^{'} )) \cdot (\xi -\xi^{'}) > 0,
	\end{aligned}
\end{equation}
for almost every $x$ in $\Omega$ and for every $\xi\neq\xi^{'}$ in $\mathbb{R}^N$.
\begin{theorem}\label{teo_uniqueness}
Let $a(x,s,\xi) = \tilde{a}(x,\xi)b(s)$ where $\tilde{a}$ satisfies \eqref{caraunique} and $b$ is a bounded and Lipschitz function satisfying \eqref{b}. Let $h$ be decreasing and satisfying \eqref{h2}. Let $f\in L^m(\Omega)$ be positive almost everywhere in $\Omega$. Then, if $\theta\le \frac{\gamma_2}{p-1}$ and $m=1$, there is at most one entropy solution to \eqref{pbmain}. Moreover, if $\theta > \frac{\gamma_2}{p-1}$ and $m>1$, there is at most one entropy solution $u$ to \eqref{pbmain} such that $u^{\frac{m(\theta(p-1)-\gamma_2)}{m-1}} \in L^1(\Omega)$. In addition, if \eqref{condent} holds and
\begin{equation}\label{condmuni}
\displaystyle m\ge \max\left(\frac{N(p-1)}{(N-p)(\gamma_2-\theta(p-1))+N(p-1)}, 1\right),
\end{equation}
then there is at most one entropy solution to \eqref{pbmain}.
\end{theorem}

Note that, under the assumption \eqref{caraunique} on $\tilde a$ and since $b$ is a bounded Lipschitz function, $a$ satisfies \eqref{cara1}, \eqref{cara2}, \eqref{cara3}. Thus Theorem \ref{teo_ent} and Theorem \ref{teo_uniqueness} take us naturally to state the following existence and uniqueness result:

\begin{Corollary}\label{corunique}
Let $a(x,s,\xi) = \tilde{a}(x,\xi)b(s)$ where $\tilde{a}$ satisfies \eqref{caraunique} and $b$ is a bounded and Lipschitz function satisfying \eqref{b} and \eqref{condent}. Let $h$ be decreasing and satisfying \eqref{h1} with $\gamma_1\le 1$ and \eqref{h2}. Let $f\in L^m(\Omega)$ be positive almost everywhere in $\Omega$ satisfying \eqref{condmuni}. Then there exists a unique entropy solution to \eqref{pbmain}.
\end{Corollary}

Condition \eqref{condmuni} gives that one needs only integrable data if $\theta\le \frac{\gamma_2}{p-1}$ to expect the existence of a unique entropy solution to the problem. Otherwise, if $\theta> \frac{\gamma_2}{p-1}$ one needs to require more summability for $f$ in order to have that the solution found in Theorem \ref{teo_ent} is unique. \\
Let us also highlight that, in the limit case $\theta= 1+\frac{\gamma_2}{p-1}$, we require $f$ in $L^{\frac{N}{p}}(\Omega)$ to have uniqueness of solutions; this seems to fit perfectly with the idea that, in this limit case and when $f$ is less regular, solutions lose their regularity and one can not expect to be in the uniqueness class given by Theorem \ref{teo_uniqueness}.

\begin{remark}
Note that in \cite{p} it is proved uniqueness of an entropy solution to \eqref{pbmain} under assumptions on $a$ and $b$ analogous to the one of the above theorem with $p=2$ and $\theta<1$. In the just cited paper, $f$ is a nonnegative function belonging to $L^m(\Omega)$ with $m\ge \frac{N(2-\theta)}{2+N(1-\theta)}$ and $h$ is constantly one. It is worth mentioning that this result should be extended to the case of a continuous, bounded and non-increasing function $h$ (i.e. $\gamma_1=\gamma_2=0$). Here, in Theorem \ref{teo_uniqueness} (and, consequently, in Corollary \ref{corunique}) when $p=2$ and $\gamma_2=0$, we require that $f>0$ almost everywhere in $\Omega$ and that $h$ is a decreasing function but we just need that $f\in L^m(\Omega)$ with $m\ge \frac{N}{N-\theta(N- 2)}$. Let observe that, if $N\ge 2$, it always holds
$$\frac{N}{N-\theta(N- 2)}\leq\frac{N(2-\theta)}{2+N(1-\theta)},$$
which implies that we prove uniqueness for less regular $f$'s.
\end{remark}

\section{Existence of entropy solutions}
\label{s4}

In order to prove Theorem \ref{teo_ent} we introduce the following approximation problem for \eqref{pbmain}:
\begin{equation}
\label{pbapprox}
\begin{cases}
\displaystyle -\operatorname{div}(a(x,T_n(u_n),\nabla u_n)) = h_n(u_n)f_n & \text{ in }\Omega,\\
u_n=0 & \text{ on } \partial\Omega,
\end{cases}
\end{equation}
where $h_n(s):= T_n(h(s))$ if $s\ge 0$ and $h_n(s):= \min(n,h(0))$ if $s<0$, $f_n:=T_n(f)$. The existence of a weak solution $u_n\in\sob \cap L^\infty(\Omega)$ to \eqref{pbapprox} is guaranteed by \cite{ll}. Therefore one has
\begin{equation}
\label{weakappr}
\int_\Omega a(x,T_n(u_n),\nabla u_n)\cdot \nabla\varphi= \int_\Omega h_n(u_n)f_n\varphi
\end{equation}
for every  $\varphi\in\sob$. Moreover, taking $\varphi=-u_n^-$ in \eqref{weakappr} and recalling that $h_n(u_n)f_n$ is nonnegative, we obtain
\begin{eqnarray*}
\lefteqn{
\frac{\alpha}{(1+n)^{\theta(p-1)}}\|u_n^-\|_{\sob}^p \leq \alpha \int_\Omega \frac{|\nabla u_n^-|^p}{(1+T_n(u_n^-))^{\theta(p-1)}}}  \\
&& \stackrel{\eqref{cara1},\eqref{b}} \leq \int_\Omega a(x,T_n(u_n),\nabla u_n)\cdot \nabla (-u_n^-) = -\int_\Omega h_n(u_n)f_n u_n^- \leq 0.
\end{eqnarray*}
It follows that $\|u_n^-\|_{\sob}=0$ which means that $u_n$ is nonnegative.

\medskip

We start by proving that $u_n$ has an almost everywhere limit in $\Omega$ as $n\to \infty$.

\begin{lemma}
\label{lem_convqo}
Let $a$ satisfy \eqref{cara1} and \eqref{b}. Let $h$ satisfy \eqref{h1} with $\gamma_1\le 1$ and \eqref{h2} and let $f\in L^1(\Omega)$ be nonnegative. Let $u_n$ be a weak solution of \eqref{pbapprox}. Then $T_k(u_n)$ is bounded in $W^{1,p}_0(\Omega)$ with respect to $n$ for every $k>0$. Moreover, if \eqref{condent} holds, then there exists a nonnegative measurable almost everywhere finite function $u$ such that $u_n$ converges to $u$ almost everywhere in $\Omega$ and $T_k(u_n)$ converges weakly to $T_k(u)$ in $\sob$ for every $k>0$ as $n\to \infty$.
\end{lemma}

\begin{proof}
	Let us take $T_k(u_n)$ ($k>0$) as a test function in \eqref{pbapprox}; then one gets that
	\begin{equation}\label{tklim}
		\int_\Omega a(x, T_n(u_n), \nabla u_n)\cdot \nabla T_k(u_n) = \int_\Omega h_n(u_n)f_nT_k(u_n).
	\end{equation}
	For the left-hand side of \eqref{tklim} we use \eqref{cara1} and \eqref{b} and observing that $T_n(u_n)\le k$ where the integrand is not null one has
	\begin{equation}\label{tklim2}
		\int_\Omega a(x, T_n(u_n), \nabla u_n)\cdot \nabla T_k(u_n)\ge \frac{\alpha}{(1+k)^{\theta(p-1)}}\int_\Omega |\nabla T_k(u_n)|^p.
	\end{equation}
	For the right-hand side of \eqref{tklim} one has
		\begin{equation}\label{tklim3}
		\int_\Omega h_n(u_n)f_nT_k(u_n) \stackrel{\eqref{h1}}\le c_1s_1^{1-\gamma_1}\int_{\{u_n\le s_1\}}f + k\left(\sup_{s\in(s_1,\infty)}h(s)\right)\int_{\{u_n > s_1\}}f.
	\end{equation}
	Hence, since by assumption $f\in L^1(\Omega)$, it follows from \eqref{tklim}, \eqref{tklim2} and \eqref{tklim3} that $T_k(u_n)$ is bounded in $W^{1,p}_0(\Omega)$ with respect to $n$.
	
	\medskip

Let now assume \eqref{condent}. We prove that, up to subsequences not relabeled, $u_n$ converges almost everywhere to a function $u$ as $n\to \infty$. In general, we can not show that $u_n$ (or some power of it) are bounded in $L^1(\Omega)$ but we can control the measure of superlevel sets. With this aim let us take $\log(1+T_k(u_n)) (1+u_n)^{\gamma_2}$
as a test function in \eqref{pbapprox}; then one has
	\begin{equation}\label{convqo1}
		\int_\Omega a(x,T_n(u_n),\nabla u_n)\cdot \nabla T_k(u_n) \frac{(1+u_n)^{\gamma_2}}{1+T_k(u_n)} \le \int_\Omega h_n(u_n)f_n \log(1+T_k(u_n)) (1+u_n)^{\gamma_2},
	\end{equation}
where we have got rid of the nonnegative term involving the gradient of $(1+u_n)^{\gamma_2}$.
\\
For the left-hand side of \eqref{convqo1} we apply  \eqref{cara1} and \eqref{b}, yielding to
\begin{equation}
	\begin{aligned}\label{convqo2}
		\int_\Omega a(x,T_n(u_n),\nabla u_n)\cdot \nabla T_k(u_n) \frac{(1+u_n)^{\gamma_2}}{1+T_k(u_n)} &\ge \alpha \int_\Omega \frac{|\nabla T_k(u_n)|^p}{(1+T_k(u_n))^{\theta(p-1)+1-\gamma_2}} \\
&\stackrel{\eqref{condent}} \geq \alpha \int_\Omega \frac{|\nabla T_k(u_n)|^p}{(1+T_k(u_n))^p} \\	
		&= \alpha \int_\Omega |\nabla \log(1+T_k(u_n))|^p.
	\end{aligned}
\end{equation}
Furthermore, for the right-hand side of \eqref{convqo1} we apply \eqref{h1} and \eqref{h2} (recall that $\gamma_1\le 1$), deducing
\begin{equation}\label{convqo3}
	\begin{aligned}
\int_\Omega h_n(u_n)f_n \log(1+T_k(u_n)) (1+u_n)^{\gamma_2} &\le c_1s_1^{1-\gamma_1}(1+s_1)^{\gamma_2}\int_{\{u_n\le s_1\}}f
\\
&+ (1+s_2)^{\gamma_2+1}\left(\sup_{s\in (s_1,s_2)}h(s) \right)\int_{\{s_1 < u_n< s_2\}} f
\\
&+ c_2\left(\frac{1+s_2}{s_2}\right)^{\gamma_2}\log(1+k)\int_{\{u_n \ge s_2\}} f,
	\end{aligned}
\end{equation}
where we use that $\log(1+s)\le s$ for any $s\geq 0$. Hence from \eqref{convqo3} one has for $k\geq e-1$
\begin{equation}\label{convqo4}
	\begin{aligned}
		\int_\Omega h_n(u_n)f_n \log(1+T_k(u_n)) (1+u_n)^{\gamma_2} \le c \|f\|_{L^1(\Omega)}\log(1+k)
	\end{aligned},
\end{equation}
where $c$ is a positive constant independent of both $n$ and $k$.
Hence gathering  \eqref{convqo2} and \eqref{convqo4} in \eqref{convqo1} one gets
\begin{equation}\label{convqo5}
	\alpha \int_\Omega |\nabla \log(1+T_k(u_n))|^p\le c\|f\|_{L^1(\Omega)}\log(1+k).
\end{equation}
Then it follows from the Sobolev inequality that
\begin{equation}\label{convqo6}
	\begin{aligned}
		\int_\Omega |\nabla \log(1+T_k(u_n))|^p &\ge \frac{1}{\mathcal{S}^p}\left(\int_\Omega (\log(1+T_k(u_n))^{p^*}\right)^{\frac{p}{p^*}}
		\\
		&\ge \frac{1}{\mathcal{S}^p} (\log(1+k))^p |\{u_n\ge k\}|^{\frac{p}{p^*}}.
	\end{aligned}
\end{equation}
Hence, combining \eqref{convqo5} and \eqref{convqo6}, one obtains that
\begin{equation}\label{superlevel}
	|\{u_n\ge k\}|\le \frac{c \|f\|_{L^1(\Omega)}^{\frac{N}{N-p}}}{(\log(1+k))^{\frac{N(p-1)}{N-p}}},
\end{equation}
which means that, as $k\to \infty$, the Lebesgue measure of the set $\{u_n\ge k\}$ tends to zero.

\smallskip

Now, since $T_k(u_n)$ is bounded in $\sob$, applying Rellich-Kondrachov Theorem, we obtain that $T_k(u_n)$ is a Cauchy sequence in $L^p(\Omega)$ for all $k>0$. Hence, up to subsequences, it is a Cauchy sequence in measure for each $k>0$. \\
Let us now show that $u_n$ is a Cauchy sequence in measure. We note that, for all $k,\eta>0$ and for all $n,m\in\N$, it holds
\begin{equation}
	\label{divido}
	\{|u_n-u_m|>\eta\}\subseteq\{u_n\geq k\}\cup\{u_m\geq k\}\cup\{|T_k(u_n)-T_k(u_m)|>\eta\}.
\end{equation}
As a consequence of \eqref{superlevel} and for any fixed $\varepsilon>0$, one gains the existence of  $\overline{k}>e-1$ such that
$$
\left|\{u_n\geq k\}\right|<\frac{\varepsilon}{3} \quad\text{and}\quad \left|\{u_m\geq k\}\right|<\frac{\varepsilon}{3} \qquad\forall n,m\in\N, \quad\forall k>\overline{k}.
$$
Moreover, using that $T_{k}(u_{n})$ is a Cauchy sequence in measure for each $k>0$ fixed, we have that there exists $n_{\varepsilon}>0$ such that
$$
\left|\{|T_k(u_n)-T_k(u_m)\right|>\eta\}|<\frac{\varepsilon}{3} \qquad\forall n,m>n_{\varepsilon}, \quad \forall\eta>0.
$$
Thus, if $k>\overline{k}$, from \eqref{divido} we obtain, for every $\varepsilon>0$, that
$$
\left|\{|u_n-u_m|>\eta\}\right|<\varepsilon\qquad\forall n,m\geq n_{\varepsilon}.
$$
This implies that $u_n$ is a Cauchy sequence in measure. Then there exists a nonnegative measurable function $u$ to which $u_n$, up to subsequences not relabeled, converges almost everywhere in $\Omega$ as $n\to\infty$. Finally, thanks to the Fatou Lemma applied to \eqref{superlevel}, it holds
$$	
|\{u\ge k\}|\le \frac{c \|f\|_{L^1(\Omega)}^{\frac{N}{N-p}}}{(\log(1+k))^{\frac{N(p-1)}{N-p}}},
$$
which implies that $u$ is almost everywhere finite in $\Omega$. Finally, since $T_k(u_n)$ is bounded in $\sob$ with respect to $n$, we deduce that $T_k(u_n)$ converges weakly to $T_k(u)$ in $\sob$ as $n\to\infty$ for every $k>0$.
\end{proof}

Note that in the following lemma we are actually using all the hypotheses we have listed in the existence Theorem \ref{teo_ent}. Indeed, also \eqref{cara2} and \eqref{cara3} will be employed.

\begin{lemma}\label{convgradienti}
Under the assumptions of Theorem \ref{teo_ent}, let $u_n$ be a weak solution of \eqref{pbapprox}. Then for any $0\le \varphi\in W^{1,p}_0(\Omega)\cap L^{\infty}(\Omega)$ it holds
	\begin{equation}\label{l1locddo}
		\int_\Omega h_n(u_n)f_n\varphi \le C,
	\end{equation}
	where $C$ does not depend on $n$. In particular
	\begin{equation}\label{stimavdeltaddo}
		\limsup_{n\to\infty}\int_{\{u_n\le \delta\}} h_n(u_n)f_n\varphi \le C_\delta,
	\end{equation}	
	where $C_\delta \to 0^+$ as $\delta \to 0^+$. Finally $\nabla u_n$ converges to $\nabla u$ almost everywhere in $\Omega$ as $n\to \infty$, where $u$ is given by Lemma \ref{lem_convqo}.
\end{lemma}

\begin{proof}
We first show \eqref{l1locddo}. Let $\varphi \in W^{1,p}_0(\Omega)\cap L^{\infty}(\Omega)$ be nonnegative and let us take $V_1(u_n)\varphi$ (where $V_1$ is defined in \eqref{Vdelta}) as a test function in the  weak formulation of \eqref{pbapprox}, yielding to	
$$
\begin{aligned}
\int_{\{u_n\le 1\}} h_n(u_n)f_n\varphi &\leq \int_\Omega a(x,T_n(u_n),\nabla u_n) \cdot \nabla \varphi V_1(u_n)
\\
&+ \int_\Omega a(x,T_n(u_n),\nabla u_n) \cdot\nabla u_n V_1'(u_n) \varphi \le C
\end{aligned}
$$
for some positive constant $C$ not dependent on $n$.
Indeed, it follows from Lemma \ref{lem_convqo} that $T_k(u_n)$ is bounded in $W^{1,p}_0(\Omega)$ with respect to $n$ for every $k>0$ and this is sufficient to deduce (recall \eqref{cara2}) that the right-hand side of the previous is bounded.
Moreover, in the set $\{u_n>1\}$, thanks to the assumptions on $h$ and $f$, the term $h_n(u_n)f_n\varphi$ is bounded in $L^1(\Omega)$ with respect to $n$. Thus \eqref{l1locddo} holds.

\medskip

Now we prove the almost everywhere convergence of the gradients. Let $0\le\varphi\in C^1_c(\Omega)$ and let us take $T_k(T_l(u_n)-T_l(u))V_l(u_n)\varphi$ (with $\frac{n}{2}>l>k$) as a test function in \eqref{pbapprox}, obtaining
	\begin{equation*}
		\begin{aligned}
			&\int_\Omega (a(x,T_l(u_n),\nabla T_l(u_n)) - a(x,T_l(u_n),\nabla T_l(u))) \cdot \nabla T_k(T_l(u_n)-T_l(u)) \varphi \\
			&= - \int_{\{l<u_n<2l\}} a(x,u_n,\nabla u_n)  \cdot \nabla T_k(T_l(u_n)-T_l(u)) \varphi V_l(u_n) \\
			&+ \frac{1}{l}\int_{\{l<u_n<2l\}} a(x,u_n,\nabla u_n)\cdot \nabla u_n T_k(T_l(u_n)-T_l(u)) \varphi \\
			& -  \int_{\Omega} a(x,u_n,\nabla u_n)  \cdot \nabla \varphi T_k(T_l(u_n)-T_l(u)) V_l(u_n) \\
			& + \int_\Omega h_n(u_n) f_n T_k(T_l(u_n)-T_l(u))\varphi V_l(u_n) \\
			& -\int_\Omega a(x,T_l(u_n),\nabla T_l(u)) \cdot \nabla T_k(T_l(u_n)-T_l(u)) \varphi =: (A)+(B)+(C)+(D)+(E).	
		\end{aligned}
	\end{equation*}
We first observe that
	$$\limsup_{n\to\infty} \ \left((C)+(E)\right) \leq  0.$$
	Indeed, for the $(C)$ term one can observe that $|a(x,u_n,\nabla u_n)|V_l(u_n)$ is bounded in $L^{\frac{p}{p-1}}(\Omega)$ with respect to $n$ and that $T_l(u_n)-T_l(u)$ converges to zero in any Lebesgue space as $n\to\infty$. For $(E)$ we use \eqref{cara2} and the weak convergence of $T_l(u_n)$ to $T_l(u)$ as $n\to\infty$ in $W^{1,p}_0(\Omega)$.
\\
For $(A)$, we note, since $T_k(T_l(u)) = T_k(u)$ for $l>k$, that
	$$(A) \le C\int_\Omega|a(x,u_n,\nabla u_{n})|V_l(u_n)|\nabla T_{k}(u)|\chi_{\{u_n>l\}},$$
	and, since $|a(x,u_n,\nabla u_{n})|V_l(u_n)$ is bounded in $L^{\frac{p}{p-1}}(\Omega)$ with respect to $n$ and since, by Lemma \ref{lem_convqo}, $\displaystyle |\nabla T_k(u)|\chi_{\{u_n>l\}}$ converges to zero in $L^p(\Omega)$, then one has
	$$
		\limsup_{n\to\infty} (A) \le 0.
	$$
For $(B)$ using \eqref{cara2} and after an application of the Young inequality, recalling that $l$ is fixed, one has
$$
	\begin{aligned}
		(B)&\le \frac{Ck}{l}\int_{\{l<u_n<2l\}} |a(x,u_n,\nabla u_n)||\nabla u_n|
		\\
		&\le
		\frac{Ck}{l}\int_{\{l<u_n<2l\}} \left(\ell^{\frac{p}{p-1}} + u_n^p + |\nabla u_n|^p\right) \le Ck.
	\end{aligned}	
$$
	Moreover, thanks to \eqref{l1locddo}, one also has that $(D)\le Ck$. Hence we have proven that
	$$\limsup_{n\to\infty} \int_\Omega (a(x,T_l(u_n),\nabla T_l(u_n)) - a(x,T_l(u_n),\nabla T_l(u))) \cdot \nabla T_k(T_l(u_n)-T_l(u)) \varphi\le Ck.$$
	This allows to reason as in the second part of Theorem $2.1$ of \cite{bm} deducing that $\nabla T_l(u_n)$ converges almost everywhere to $\nabla T_l(u)$ in $\Omega$ as $n\to\infty$ for every $l>0$. Thus the result follows.

	\medskip

	Now in order to show \eqref{stimavdeltaddo}, for $\delta>0$ sufficiently small let us take $V_\delta(u_n)\varphi$ as a test function in \eqref{pbapprox} obtaining (recall that $V'_{\delta}(s) \le 0$ for any $s\ge 0$)
	\begin{equation*}
		\int_{\{u_n\le \delta\}} h_n(u_n)f_n \varphi \le \int_\Omega h_n(u_n)f_n V_\delta(u_n)\varphi  \le \int_\Omega a(x,u_n,\nabla u_n)\cdot \nabla \varphi V_\delta(u_n).
	\end{equation*}
Then taking $n\to\infty$, thanks to the almost everywhere convergence of gradients just proven and since $a(x,u_n,\nabla u_n)V_\delta(u_n)$ is bounded in $L^{\frac{p}{p-1}}(\Omega)^N$, one has that $a(x,u_n,\nabla u_n)V_\delta(u_n)$ converges weakly to $a(x,u,\nabla u)V_\delta(u)$ in $L^{\frac{p}{p-1}}(\Omega)^N$. Therefore
	\begin{equation*}
		\begin{aligned}
			\limsup_{n\to\infty} \int_{\{u_n\le \delta\}} h_n(u_n) f_n \varphi &\le \int_\Omega a(x,u,\nabla u)\cdot \nabla \varphi V_\delta(u) = C_\delta.
		\end{aligned}
	\end{equation*}
Now let us simply observe that
$$ \lim_{\delta\to 0^+} C_\delta = \int_{\{u=0\}} a(x,u,\nabla u) \cdot \nabla \varphi,$$
and, since it follows from \eqref{cara1} that $a(x,0,0)=0$ for almost every $x \in \Omega$, we obtain the result.
\end{proof}

In order to pass to the limit in \eqref{pbapprox} as $n\to \infty$, we need a stronger convergence result for the truncations of the solutions to \eqref{pbapprox}.

\begin{lemma}
\label{stronconv}
Under the assumptions of Theorem \ref{teo_ent}, let $u_n$ be a weak solution of \eqref{pbapprox}. Then $T_k(u_n)$ converges to $T_k(u)$ in $\sob$ as $n\to \infty$ for every $k>0$, where $u$ is given by Lemma \ref{lem_convqo}.
\end{lemma}

\begin{proof}
The result will follow from an application of Lemma $5$ in \cite{bmp} once we show that for any $k>0$	
\begin{equation}\label{strong}
\lim_{n\to\infty}\int_\Omega\big(a(x,T_k(u_n),\nabla T_{k}(u_{n}))-a(x,T_k(u_n),\nabla T_{k}(u))\big)\cdot\nabla(T_{k}(u_{n})-T_{k}(u))=0.
\end{equation}
 Let us remark that, without loss of generality, we assume $n>k$.
	
\smallskip
	
With this aim, in the weak formulation of \eqref{pbapprox}, we take $(T_{k}(u_{n})-T_{k}(u))V_l(u_n)$ (with $l>k$) as a test function obtaining

	\begin{equation}
		\begin{aligned}\label{1-4.1}
			&\int_\Omega \big(a(x,T_k(u_n),\nabla T_{k}(u_{n}))-a(x,T_k(u_n),\nabla T_{k}(u))\big)\cdot\nabla(T_{k}(u_{n})-T_{k}(u))\\
			=&-\int_{\{k<u_n<2l\}}a(x,T_n(u_n),\nabla u_{n})\cdot\nabla(T_{k}(u_{n})-T_{k}(u))V_l(u_n) \ \ \  \\
			&+\frac{1}{l}\int_{\{l<u_n<2l\}}a(x,T_n(u_n),\nabla u_{n})\cdot\nabla u_{n}(T_{k}(u_{n})-T_{k}(u)) \ \ \ \\
			&+\int_\Omega h_n(u_n)f_n(T_{k}(u_{n})-T_{k}(u))V_l(u_n) \ \ \ \\
			&-\int_\Omega a(x,T_k(u_n),\nabla T_{k}(u))\cdot\nabla(T_{k}(u_{n})-T_{k}(u))
			=: (A)+(B)+(C)+(D).\ \ \
		\end{aligned}
	\end{equation}
For $(A)$, we note that
$$(A) \le C\int_\Omega|a(x,T_n(u_n),\nabla u_{n})|V_l(u_n)|\nabla T_{k}(u)|\chi_{\{u_n>k\}}.$$
Since, by \eqref{cara2} and Lemma \ref{lem_convqo}, $|a(x,T_n(u_n),\nabla u_{n})|V_l(u_n)$ is bounded in $L^{\frac{p}{p-1}}(\Omega)$ with respect to $n$ and  $\displaystyle |\nabla T_k(u)|\chi_{\{u_n>k\}}$ converges to zero in $L^p(\Omega)$, then one has
	\begin{equation}\label{Aterm}
		\limsup_{n\to\infty} (A)\leq 0.
	\end{equation}
In order to estimate $(B)$ we take $1-V_l(u_n)$ as a test function in \eqref{pbapprox}. One yields to
\begin{equation*}
\begin{aligned}
\frac{1}{l}\int_{\{l<u_n<2l\}}a(x,T_n(u_n),\nabla u_{n})\cdot\nabla u_{n}&=\int_\Omega h_n(u_n)f_n(1-V_l(u_n))
\\
&\le \left(\sup_{s\in [l,\infty)}h(s)\right)\int_\Omega f (1-V_l(u_n)), \ \ \  \\
\end{aligned}\end{equation*}
and the right-hand side of the previous, by a double application of Lebesgue Theorem, goes to zero as $n\to\infty$ and $l\to\infty$. Whence, since
$$(B)\le \frac{2k}{l}\int_{\{l<u_n<2l\}}a(x,T_n(u_n),\nabla u_{n})\cdot\nabla u_{n},$$
we deduce
	\begin{equation}\label{Bterm}
		\limsup_{l\to\infty}\limsup_{n\to\infty} (B) \leq 0.
	\end{equation}
For the $(C)$ term we note that, for $\delta$ small enough, we have
	\begin{equation}
	\label{stapp}
		(C) \stackrel{\eqref{h1}}\le C\delta^{1-\gamma_1}\int_{\{u_n \le \delta\}} f +  \int_{\{u_n> \delta\}} h_n(u_n)f_n (T_{k}(u_{n})-T_{k}(u))V_l(u_n).
		\end{equation}
Thus, if $h(0)<\infty$ (i.e. $\gamma_1=0$), the first term on the right-hand side of \eqref{stapp} converges to $0$ as $n\to \infty$ and $\delta\to 0^+$. Otherwise, if $h(0)=\infty$ (i.e. $0<\gamma_1\leq 1$), applying Fatou Lemma to \eqref{l1locddo} we deduce that $h(u)f\in L^1_{\rm loc}(\Omega)$, and so, $\{u=0\}\subset \{f=0\}$ up to a set of zero Lebesgue measure. It follows that, in both cases, we have
\begin{equation}
\label{stapp2}
\limsup_{\delta \to 0^+} \limsup_{n\to\infty} \left( \delta^{1-\gamma_1}\int_{\{u_n \le \delta\}} f\right)= \limsup_{\delta \to 0^+} \left(\delta^{1-\gamma_1}\int_{\{u \le \delta\}} f\right) =0.
\end{equation}
One also observes that, applying Lebesgue Theorem, we obtain
\begin{eqnarray}
\label{stapp3}
&&\limsup_{n\to\infty}\int_{\{u_n> \delta\}} h_n(u_n)f_n(T_{k}(u_{n})-T_{k}(u))V_l(u_n) \nonumber \\
&&\le \lim_{n\to\infty}\left(\sup_{s\in [\delta,\infty)}h(s)\right) \int_\Omega f|T_{k}(u_{n})-T_{k}(u)|=0.
\end{eqnarray}
Hence, it follows from \eqref{stapp}, \eqref{stapp2} and \eqref{stapp3} that	
\begin{equation}\label{Cterm}
		\limsup_{\delta\to 0^+} \limsup_{l\to\infty}\limsup_{n\to\infty} (C)\leq 0.
	\end{equation}	
As regards $(D)$, by Lemma \ref{lem_convqo} and using \eqref{cara2}, $a(x,T_k(u_n),\nabla T_k(u))$ strongly converges to $a(x,T_k(u),\nabla T_k(u))$ in $L^{\frac{p}{p-1}}(\Omega)^N$ and $\nabla (T_k(u_n)- T_k(u))$ weakly converges to $0$ in $L^p(\Omega)^N$, whence
\begin{equation}
\label{Dterm}
\lim_{n\to\infty} (D)= 0.
\end{equation}
Finally, gathering \eqref{Aterm},\eqref{Bterm}, \eqref{Cterm} and \eqref{Dterm} into \eqref{1-4.1} and recalling \eqref{cara3}, we deduce \eqref{strong}.

\end{proof}

Now we have all the tools in order to show that $u$, given by Lemma \ref{lem_convqo}, is an entropy solution to \eqref{pbmain}.

\begin{proof}[Proof of Theorem \ref{teo_ent}]
Let $u_n$ be a weak solution to \eqref{pbapprox}. By Lemma \ref{lem_convqo}, its almost everywhere limit $u$ as $n\to\infty$ exists, it is almost everywhere finite in $\Omega$ and  $T_k(u)\in\sob$ for every $k>0$.

\medskip

First we prove \eqref{ent0}. One simply deduces that $a(x,T_k(u),\nabla T_k(u)) \in L^{\frac{p}{p-1}}(\Omega)^N$ from \eqref{cara2} and from the fact that $T_k(u)\in\sob$ for every $k>0$. Now we focus on proving $h(u)fT_k(u-\varphi) \in L^1(\Omega)$ for any $\varphi\in W^{1,p}_0(\Omega) \cap L^\infty(\Omega)$. Taking $T_k(u_n-\varphi)^+$ as test function in \eqref{pbapprox} and defining $M:= ||\varphi||_{L^\infty(\Omega)} + k$, we have
$$
\begin{aligned}
&\int_\Omega h_n(u_n)f_n T_k(u_n-\varphi)^+ = \int_\Omega a(x,T_n(u_n),\nabla u_n) \cdot \nabla T_k(u_n-\varphi)^+
\\
&\le \int_{\{\varphi<u_n<\varphi+ k\}} |a(x,T_M(u_n),\nabla T_M(u_n)) \cdot \nabla (T_M(u_n)-\varphi)|
\\
&\stackrel{\eqref{cara2}}{\le} \beta \int_{\{\varphi<u_n<\varphi+ k\}} (\ell(x) + M^{p-1}+|\nabla T_M(u_n)|^{p-1}) (|\nabla T_M(u_n)| + |\nabla \varphi|)\le C,
\end{aligned}
$$
where $C$ does not depend on $n$ thanks to Lemma \ref{lem_convqo}. Then an application of the Fatou Lemma gives that $h(u)f T_k(u-\varphi)^+ \in L^1(\Omega)$. A similar argument with $T_k(u_n-\varphi)^-$ gives that $h(u)f T_k(u-\varphi) \in L^1(\Omega)$.

\medskip

Now we prove \eqref{ent}. Hence let $\varphi \in W^{1,p}_0(\Omega)\cap L^\infty(\Omega)$ and let us take $T_k(u_n-\varphi)$ as a test function in the weak formulation of \eqref{pbapprox}, yielding to
\begin{equation}
\label{lim1}
\int_{\Omega} a(x,T_n(u_n),\nabla u_n)\cdot\nabla T_k(u_n-\varphi) = \int_{\Omega}h_n(u_n)f_nT_k(u_n-\varphi).
\end{equation}
Therefore the aim is to pass to the limit as $n\to\infty$ in \eqref{lim1}. Note that the integrand in the left-hand side is zero where $\{|u_n-\varphi|>k\}$. Then, defining $M= ||\varphi||_{L^\infty(\Omega)} + k$, we have that, by Lemma \ref{stronconv}, $T_M(u_n)$ converges strongly to $T_M(u)$ in $\sob$ as $n\to\infty$. Hence, by \eqref{cara2}, we deduce that $a(x,T_M(u_n),\nabla T_M(u_n))$ converges strongly to $a(x,T_M(u),\nabla T_M(u))$ in $L^{\frac{p}{p-1}}(\Omega)^N$ as $n\to\infty$. This implies that
\begin{equation}\label{limagg}
\lim_{n\to\infty} \int_{\Omega} a(x,T_n(u_n),\nabla u_n)\cdot\nabla T_k(u_n-\varphi) =  \int_{\Omega} a(x,u,\nabla u)\cdot\nabla T_k(u-\varphi).
\end{equation}

Let us focus on the right-hand side of \eqref{lim1}. If $h(0)<\infty$ then one can simply pass to the limit as $n\to\infty$ applying Lebesgue Theorem. Hence, without loss of generality, we assume that $h(0)=\infty$ and we split in the following two terms
\begin{equation*}
	\int_{\Omega} h_n(u_n)f_nT_k(u_n-\varphi) = \int_{\Omega} h_n(u_n)f_nT_k(u_n-\varphi)^+ - \int_{\Omega} h_n(u_n)f_nT_k(u_n-\varphi)^-.
\end{equation*}
Now we pass to the limit the first term on the right-hand side of the previous. The second term can be treat in an analogous way. \\
Let $\delta$ be such that $\delta\not\in \{\tau: |\{u=\tau\}|>0\}$ (which is a countable set). Taking $(u_n-\varphi)^+V_\delta(u_n)$  as a test function in the weak formulation of \eqref{pbapprox} and dropping a non-positive term, we obtain
\begin{equation*}
	\int_{\{u_n\le \delta\}} h_n(u_n)f_n(u_n-\varphi)^+ \le  \int_\Omega a(x,T_n(u_n), \nabla u_n)\cdot \nabla (u_n-\varphi)^+ V_\delta(u_n).
\end{equation*}
Since $T_{2\delta}(u_n)$ strongly converges in $W^{1,p}_0(\Omega)$ to $T_{2\delta}(u)$ as $n\to\infty$ then one yields to
\begin{equation}\label{lim-1}
	\limsup_{n\to\infty}\int_{\{u_n\le \delta\}} h_n(u_n)f_n(u_n-\varphi)^+ \le  \int_\Omega a(x,u, \nabla u)\cdot \nabla (u-\varphi)^+ V_\delta(u) =: C_\delta,
\end{equation}
where  $C_\delta$ goes to zero as $\delta \to 0^+$ (by the same argument used for \eqref{stimavdeltaddo}). Therefore we split once more as
\begin{equation}
	\begin{aligned}
\label{lim5}
\int_{\Omega}h_n(u_n)f_n T_k(u_n-\varphi)^+ &= \int_{\{u_n\leq \delta\}} h_n(u_n)f_n (u_n-\varphi)^+ \\
&+ \int_{\{u_n > \delta\}} h_n(u_n)f_n T_k(u_n-\varphi)^+.
	\end{aligned}
\end{equation}
for $k\geq \delta$. As regard the second term on the right-hand side of \eqref{lim5}, we have that
$$
h_n(u_n)f_nT_k(u_n-\varphi)^+\leq k \left(\sup_{s\in (\delta,\infty)}h(s)\right)f,
$$
so, using the convergence results in Lemma \ref{lem_convqo} and applying Lebesgue Theorem, we deduce
\begin{equation*}
\lim_{n\to\infty} \int_{\{u_n > \delta\}} h_n(u_n)f_nT_k(u_n-\varphi)^+ = \int_{\{u > \delta\}} h(u)fT_k(u-\varphi)^+.
\end{equation*}
Moreover, we have already proved \eqref{ent0}, thus $h(u)fT_k(u-\varphi)^+ \in L^1(\Omega)$. Hence, a second application of the Lebesgue Theorem gives that
\begin{equation}
\label{lim6}
	\lim_{\delta\to 0^+}\lim_{n\to\infty} \int_{\{u_n > \delta\}} h_n(u_n)f_nT_k(u_n-\varphi)^+ = \int_{\Omega} h(u)fT_k(u-\varphi)^+,
\end{equation}
since it follows from $h(u)fT_k(u-\varphi) \in L^1(\Omega)$ that $\{u=0\} \subset \{f=0\}$.Therefore, taking $n\to\infty$ and then $\delta \to 0^+$ in \eqref{lim5} and using \eqref{lim-1} and \eqref{lim6}, one gets that
\begin{equation*}
	\lim_{n\to\infty} \int_{\Omega} h_n(u_n)f_nT_k(u_n-\varphi)^+ = \int_{\Omega} h(u)fT_k(u-\varphi)^+.
\end{equation*}
As already mentioned, one can reason in the same way to deduce that
\begin{equation*}
	\lim_{n\to\infty} \int_{\Omega} h_n(u_n)f_nT_k(u_n-\varphi)^- = \int_{\Omega} h(u)fT_k(u-\varphi)^-,
\end{equation*}
which is sufficient to pass to the limit on the right-hand side of \eqref{lim1} and, recalling \eqref{limagg}, to obtain the result.
\end{proof}

\section{Regularity of entropy solutions}
\label{sec_reg}

In this section we show the Lebesgue regularity of any given entropy solution and of its gradient with respect to the regularity of $f$. We first provide some technical lemmas which will be useful for our purposes. Then we show that in case $\theta<1+\frac{\gamma_2}{p-1}$ one always has that a power of the solution is integrable. On the other hand, if $\theta=1+\frac{\gamma_2}{p-1}$, the Lebesgue regularity will only be recovered for regular enough $f$'s. We also highlight that, as we will see, there is continuity between these two cases.

\subsection{Auxiliary lemmas}
We give two technical lemmas. The first result is the following:
\begin{lemma}
	\label{P-test}
	Let $a$ satisfy \eqref{cara1} and let $f\in L^1(\Omega)$ be nonnegative. Let $u$ be an entropy solution to \eqref{pbmain} and let $P:\R \to \R$ be a function such that $P(u)\in \sob\cap L^{\infty}(\Omega)$ and $a(x,u,\nabla u)\cdot \nabla P(u)\geq 0$. Then
	\begin{equation}
		\label{P-formu}
		\int_\Omega a(x,u,\nabla u)\cdot \nabla P(u) \leq \int_{\Omega} h(u)fP(u).
	\end{equation}
	In particular, if $a$ satisfies \eqref{b} and $h$ satisfies \eqref{h2}, then for any $\sigma,\eta>0$ it holds
	\begin{equation}\label{d14}
		\int_{\Omega} |\nabla ((1+T_\sigma(u))^{\overline{\eta}}-1)|^p \leq C\left(1+ \int_{\{u> s_2\}} fT_{\sigma}(u)^{\eta-\gamma_2}\right),
	\end{equation}
	where $C$ is a positive constant independent of $\sigma$ and $\displaystyle \overline{\eta}:=\frac{\eta + (p-1)(1-\theta)}{p}$.
\end{lemma}

\begin{proof}
Let us firstly prove \eqref{P-formu}. Note that for $l>0$, since, by Definition \ref{defent}, $T_l(u)\in W^{1,p}_0(\Omega)\cap L^\infty(\Omega)$, one has
	$
	\varphi=T_l(u)-P(u) \in W_{0}^{1,p}(\Omega)\cap L^\infty(\Omega).
	$
	Hence, we can choose $\varphi$ as test function in \eqref{ent} with
	$
	k>\|P(u)\|_{L^{\infty}([0,\infty))},
	$
	obtaining
	\begin{equation}
		\label{P-1}
		\int_{\Omega} a(x,u,\nabla u)\cdot\nabla T_k(u-T_l(u)+P(u)) \le \int_{\Omega}h(u)fT_k(u-T_l(u)+P(u)).
	\end{equation}
	As regards the right-hand side of \eqref{P-1}, using that $0\leq P(u)< k$, we deduce that $h(u)fT_k(u-T_l(u)+P(u))$ converges to $h(u)fP(u)$ almost everywhere in $\Omega$ as $l\to \infty$. Moreover $h(u)fT_k(u-T_l(u)+P(u))\leq h(u)fT_k(u+P(u))$ and, by \eqref{ent0}$_2$, $h(u)fT_k(u+P(u))\in L^1(\Omega)$, so that one can apply the Lebesgue Theorem yielding to
	\begin{equation}
		\label{P-2}
		\lim_{l\to \infty} \int_{\Omega}h(u)fT_k(u-T_l(u)+P(u))=\int_{\Omega} h(u)fP(u).
	\end{equation}
	Now we focus on the left-hand side of \eqref{P-1}. Again, since $P(u)<k$, by \eqref{defTk}, we deduce
	\begin{equation}
		\begin{aligned}
			\label{P-3}
			\int_{\Omega} a(x,u,\nabla u)\cdot\nabla T_k(u-&T_l(u)+P(u))  = \int_{\{u\leq l\}} a(x,u,\nabla u)\cdot \nabla P(u) \\
			&+ \int_{\{u>l\}\cap\{u-T_l(u)+P(u)\leq k\}} a(x,u,\nabla u)\cdot \nabla u  \\
			&+ \int_{\{u>l\}\cap\{u-T_l(u)+P(u)\leq k\}} a(x,u,\nabla u)\cdot \nabla P(u)   \\
			&\stackrel{\eqref{cara1}} \geq  \int_{\{u\leq l\}} a(x,u,\nabla u)\cdot \nabla P(u),
		\end{aligned}
	\end{equation}
	where in the last inequality we also use that, by assumption, $a(x,u,\nabla u)\cdot \nabla P(u)\geq 0$. Hence, applying the Fatou Lemma to \eqref{P-3}, we have
	\begin{equation}
		\begin{aligned}
			\label{P-4}
			&\liminf_{l\to \infty} \int_{\Omega} a(x,u,\nabla u)\cdot\nabla T_k(u-T_l(u)+P(u))
			\\& \geq  \liminf_{l\to \infty} \int_{\{u\leq l\}} a(x,u,\nabla u)\cdot \nabla P(u)
			\\& \geq  \int_{\Omega} a(x,u,\nabla u)\cdot \nabla P(u).
		\end{aligned}
	\end{equation}
	Therefore, letting $l$ go to infinity in \eqref{P-1} and using \eqref{P-4} and \eqref{P-2}, we obtain \eqref{P-formu}.

	\medskip

	Now we focus on proving \eqref{d14}. Let us consider
	$$
	P(u)=(T_\sigma(u)+1)^{\eta}-1,
	$$
	with $\sigma,\eta>0$. Hence, by Definition \ref{defent}, we have that $P(u)\in \sob\cap L^\infty(\Omega)$, and, by \eqref{cara1} and \eqref{defTk}, that $a(x,u,\nabla u)\cdot \nabla P(u)\geq 0$, where
	\begin{equation}
		\label{piccolaere}
		\nabla P(u)=\eta(T_\sigma(u)+1)^{\eta-1}\nabla T_\sigma(u) \in L^p(\Omega).
	\end{equation}
	Therefore it follows from \eqref{P-formu} that
	\begin{equation}
		\begin{aligned}
			\label{d14bis}
			\int_{\Omega} a(x,u,\nabla u)\cdot \nabla ((T_\sigma(u)+1)^{\eta}-1) &\leq  \int_{\Omega} h(u)f((T_\sigma(u)+1)^{\eta}-1) \\
			&\stackrel{\eqref{as1},\eqref{h2},\eqref{as3}}\leq  \underline c\int_{\{u\leq s_2\}} h(u)f T_{s_2}(u)  \\
			& + \overline c c_2\int_{\{u> s_2\}} fT_{\sigma}(u)^{\eta-\gamma_2}  \\
			&\stackrel{\eqref{ent0}}\leq  C \left(1+ \int_{\{u> s_2\}} fT_{\sigma}(u)^{\eta-\gamma_2}\right).
		\end{aligned}
	\end{equation}
	Moreover, we have
	\begin{equation}
		\begin{aligned}
			\label{d14tris}
			& \int_{\Omega} a(x,u,\nabla u)\cdot \nabla ((T_\sigma(u)+1)^{\eta}-1) \\
			& \stackrel{\eqref{piccolaere}}=\eta \int_{\Omega} a(x,T_{\sigma}(u),\nabla T_\sigma(u))\cdot \nabla T_\sigma(u) (T_\sigma(u)+1)^{\eta-1} \\
			& \stackrel{\eqref{cara1},\eqref{b}}\geq C  \int_{\Omega} |\nabla( (1+T_\sigma(u))^{\overline{\eta}}-1)|^p,
		\end{aligned}
	\end{equation}
	where $\displaystyle \overline{\eta}:=\frac{\eta + (p-1)(1-\theta)}{p}$. Thus the result follows from \eqref{d14bis} and \eqref{d14tris}.
\end{proof}

We state and prove the second auxiliary result:
\begin{lemma}
	\label{aux}
	Let $1\leq p <N$. Let $v$ be a measurable function such that $T_k(v)\in \sob$ for every $k>0$. Assume that there exist $c_1>0$ and $k_0>0$ not depending on $v$ such that
	\begin{equation}
		\label{aux1}
		\int_\Omega |\nabla T_k(v)|^p \leq c_1k^\eta \qquad \forall k\geq k_0,
	\end{equation}
	for some $0\leq \eta<p$. Then $v\in M^{\frac{N(p-\eta)}{N-p}}(\Omega)$ and
	\begin{equation}
		\label{aux2}
		|\{v\geq k\}|\leq \frac{(c_1\mathcal{S}^{p})^{\frac{N}{N-p}}}{k^\frac{N(p-\eta)}{N-p}} \qquad\forall k\geq k_0.
	\end{equation}
	Moreover $|\nabla v|\in M^{\frac{N(p-\eta)}{N-\eta}}(\Omega)$ and there exist $c_2>0$ and $\lambda_0>0$ not dependent on $v$ such that
	\begin{equation}
		\label{aux6}
		|\{|\nabla v|\geq \lambda\}|\leq \frac{c_2}{\lambda^{\frac{N(p-\eta)}{N-\eta}}} \qquad \forall \lambda \geq \lambda_0.
	\end{equation}
\end{lemma}
\begin{proof}
	Let's start by proving that, if \eqref{aux1} holds, then $v\in M^{\frac{N(p-\eta)}{N-p}}(\Omega)$. Applying the Sobolev inequality to the left-hand side of \eqref{aux1}, we obtain, for every $k\geq k_0$, that
	$$
	\mathcal{S}^{-p}k^{p}|\{v\geq k\}|^{\frac{N-p}{N}}=\mathcal{S}^{-p}\left(\int_{\{v\geq k\}}k^{p^*}\right)^{\frac{p}{p^*}}\stackrel{\eqref{defTk}}\leq\mathcal{S}^{-p}\left(\int_\Omega T_k(v)^{p^*}\right)^{\frac{p}{p^*}}\leq c_1k^\eta.
	$$
	It follows \eqref{aux2} and, so, $v\in M^{\frac{N(p-\eta)}{N-p}}(\Omega)$. \\
	Now we focus on $\nabla v$. Let $\lambda>0$ and $k\geq k_0$. We have
	\begin{equation}
		\label{aux3}
		\begin{aligned}
			|\{|\nabla v|\geq \lambda\}|&=|\{|\nabla v|\geq \lambda, |v|<k\}| + |\{|\nabla v|
			\geq \lambda, |v|\geq k\}|
			\\
			&\leq |\{|\nabla T_k(v)|\geq \lambda\}| + |\{|v|\geq k\}|.
		\end{aligned}	
	\end{equation}
	As regard the first addend on the right-hand side of \eqref{aux3}, we have
	\begin{equation}
		\label{aux4}
		|\{|\nabla T_k(v)|\geq \lambda\}|=\int_{\{|\nabla T_k(v)|\geq \lambda\}} 1\leq \int_\Omega \frac{|\nabla T_k(v)|^p}{\lambda^p}\stackrel{\eqref{aux1}}\leq  \frac{c_1k^{\eta}}{\lambda^p}.
	\end{equation}
	Inserting \eqref{aux2} and \eqref{aux4} in \eqref{aux3}, we deduce
	\begin{equation}
		\label{aux5}
		|\{|\nabla v|\geq \lambda\}|\leq  \frac{c_1k^{\eta}}{\lambda^p} + \frac{(c_1\mathcal{S}^p)^{\frac{N}{N-p}}}{k^\frac{N(p-\eta)}{N-p}} \qquad\forall \lambda>0 \text{ and }k\geq k_0.
	\end{equation}
	Minimizing the function on the right-hand side of \eqref{aux5} with respect to $k$, we obtain that the minimizer is attained at
	$$
	k=C\lambda^{\frac{N-p}{N-\eta}},
	$$
	where $C$ depends on $N, p, \eta, c_1, \mathcal{S}$. Furthermore, there exists $\lambda_0>0$ such that $C\lambda^{\frac{N-p}{N-\eta}}\geq k_0$, for every $\lambda\geq \lambda_0$. Hence, choosing $\lambda\geq \lambda_0$ and evaluating the right-hand side of \eqref{aux5} at this point of minimum, we obtain \eqref{aux6}. It follows that $|\nabla v|\in M^{\frac{N(p-\eta)}{N-\eta}}(\Omega)$.
\end{proof}

\subsection{Unbounded solutions}

We first state and prove two regularity results if $\displaystyle \theta<1+\frac{\gamma_2}{p-1}$. The first one is the following:

\begin{lemma}
\label{lem_sum_m=1}
Let $a$ satisfy \eqref{cara1} and \eqref{b} with $\theta<1+\frac{\gamma_2}{p-1}$. Let $h$ satisfy \eqref{h2} and let $f\in L^1(\Omega)$ be nonnegative. Then any entropy solution $u$ to \eqref{pbmain} is such that:
	\begin{itemize}
		\item[$(i)$] if $\gamma_2\geq 1$ and $0\leq \theta \leq \frac{\gamma_2-1}{p-1}$, then $u\in\sob$; \\
		\item[$(ii)$] if $\max\left(0,\frac{\gamma_2-1}{p-1}\right) \leq \theta <1+\frac{\gamma_2}{p-1}$, then $u\in M^t(\Omega)$ and $|\nabla u|\in M^r(\Omega)$, where
		\begin{equation*}
			t=\frac{N((p-1)(1-\theta)+\gamma_2)}{N-p} \quad \text{and} \quad r=\frac{N((p-1)(1-\theta)+\gamma_2)}{N-\theta(p-1)-1+\gamma_2}.
		\end{equation*}
	\end{itemize}
\end{lemma}

\begin{remark}
	\label{rem1}
	Let us highlight that in the previous lemma one gets that
	$$
	r>1 \ \ \text{if and only if} \ \ 0 \leq \theta < \frac{N}{N-1}+\frac{\gamma_2-1}{p-1}.
	$$
	Hence let us observe that, since $p<N$, we have
	$$
	\max\left(0,\frac{\gamma_2-1}{p-1}\right)<\frac{N}{N-1}+\frac{\gamma_2-1}{p-1}<1+\frac{\gamma_2}{p-1}.
	$$
	It follows that, if
	\begin{equation*}
		\max\left(0,\frac{\gamma_2-1}{p-1}\right)\leq \theta< \frac{N}{N-1}+\frac{\gamma_2-1}{p-1},
	\end{equation*}
	then $u$ belongs to $W^{1,q}_0(\Omega)$, for every $q<r$. \\
	We also underline that
	$$
	r>p-1 \qquad\text{if and only if}\qquad 0\leq \theta < \frac{1}{N-p+1} + \frac{\gamma_2}{p-1}.
	$$
	Note that, once again, since $p<N$,
	$$
	\max\left(0,\frac{\gamma_2-1}{p-1}\right)<\frac{1}{N-p+1} + \frac{\gamma_2}{p-1}<1+\frac{\gamma_2}{p-1}.
	$$
\end{remark}

If $f$ is more than integrable, namely if $1<m<\frac{N}{p}$, we have the following regularity properties for any entropy solution to \eqref{pbmain}.

\begin{lemma}\label{lem_sum_m>1}
Under the assumptions of Lemma \ref{lem_sum_m=1}, assume that $f\in L^m(\Omega)$ with $m\geq 1$. If $1<m<\frac{N}{p}$, then any entropy solution $u$ to \eqref{pbmain} is such that $u^{\frac{Nm((p -1)(1-\theta)+\gamma_2)}{N-mp}} \in L^1(\Omega)$. Moreover the following hold:
\begin{itemize}
\item[$(i)$] if $m \ge \max\left(\frac{p^*}{p^*-\theta(p-1)-1+\gamma_2},1\right)$, then $u \in W^{1,p}_0(\Omega)$;
\item[$(ii)$]	if $1< m < \frac{p^*}{p^*-\theta(p-1)-1+\gamma_2}$, then $|\nabla u|^\frac{Nm((p-1)(1-\theta)+\gamma_2)}{N-m(\theta(p-1)+1-\gamma_2)} \in L^1(\Omega)$.
	\end{itemize}
\end{lemma}

Note that $\frac{Nm((p -1)(1-\theta)+\gamma_2)}{N-mp}$ and $\frac{Nm((p-1)(1-\theta)+\gamma_2)}{N-m(\theta(p-1)+1-\gamma_2)}$ (the exponents of Lemma \ref{lem_sum_m>1}) converge to $t$ and $r$ (the exponents of Lemma \ref{lem_sum_m=1}) respectively as $m$ tends to $1$.\\
Moreover, the previous converge both to $0$ as $\theta$ converges to $1+\frac{\gamma_2}{p-1}$ for every $1\leq m<\frac{N}{p}$. On the other hand, $\frac{p^*}{p^*-\theta(p-1)-1+\gamma_2}$ converges to $\frac{N}{p}$ as $\theta$ converges to $1+\frac{\gamma_2}{p-1}$. This seems to suggest that one can find entropy solutions with finite energy for $m\geq \frac{N}{p}$ if $\theta=1+\frac{\gamma_2}{p-1}$; in particular this will be the content of Lemma \ref{crit^2} below. By the way, it also leads to the idea that solutions should not have any Marcinkiewicz regularity when $1\leq m<\frac{N}{p}$ and $\theta=1+\frac{\gamma_2}{p-1}$.

\medskip

Now we are ready to prove the regularity lemmas. We start by proving the case with $f$ merely integrable.

\begin{proof}[Proof of Lemma \ref{lem_sum_m=1}]
Let us apply \eqref{d14} with $\eta=\theta(p-1)+1$ obtaining that
	\begin{equation}\label{somml1}
			\int_\Omega |\nabla T_\sigma(u)|^p \le C\left(1+ \int_{\{u> s_2\}} fT_{\sigma}(u)^{\theta(p-1)+1-\gamma_2}\right),
	\end{equation}
for every $\sigma>0$.

\smallskip

	\textit{Proof of $(i)$.} In this case one deduces from \eqref{somml1} that
	\begin{equation*}
		\int_\Omega |\nabla T_\sigma(u)|^p \leq C,
	\end{equation*}
where $C$ is a constant independent of $\sigma$.
Hence, after an application of the Fatou Lemma as $\sigma \to \infty$, since $u$ is almost everywhere finite on $\Omega$, we obtain $(i)$.

\smallskip

\textit{Proof of $(ii)$.}	It follows from \eqref{somml1} that
			\begin{equation*}
				\int_\Omega |\nabla T_\sigma(u)|^p \leq C \sigma^{\theta(p-1)+1-\gamma_2}  \qquad \forall \sigma>0,
			\end{equation*}
		where, once again, $C$ is a constant independent of $\sigma$. Hence, since by assumption $\theta(p-1)+1-\gamma_2<p$, applying Lemma \ref{aux}, we obtain the result.
\end{proof}

Here we prove Lebesgue regularity for an entropy solution and its gradient when $f$ is more than integrable, namely $1<m<\frac{N}{p}$.
\begin{proof}[Proof of Lemma \ref{lem_sum_m>1}]

We first prove the Lebesgue regularity for $u$. Let $\eta>\gamma_2$ and let us apply the Sobolev inequality and the H\"older inequality respectively to the left-hand and to the right-hand side of \eqref{d14}, yielding to
		\begin{equation*}
			\begin{aligned}
				&\frac{1}{\mathcal{S}^p} \left(\int_\Omega \left((1 + T_\sigma(u))^{\frac{\eta + (p-1)(1-\theta)}{p}}-1\right)^{p^*} \right)^{\frac{p}{p^*}}
				\\
				&\le C + C \|f\|_{L^m(\Omega)} \left(\int_{\{u> s_2\}} u^{(\eta-\gamma_2)\frac{m}{m-1}}\right)^{\frac{m-1}{m}}.	
			\end{aligned}
		\end{equation*}
Now taking the liminf as $\sigma\to \infty$ in the previous, thanks to the Fatou Lemma (recall that, by Definition \ref{defent}, $u$ is almost everywhere finite in $\Omega$), we obtain that
		\begin{equation}
			\begin{aligned}\label{regsomm}
				&\frac{1}{\mathcal{S}^p} \left(\int_\Omega \left((1 + u)^{\frac{\eta + (p-1)(1-\theta)}{p}}-1\right)^{p^*} \right)^{\frac{p}{p^*}}
				\\
				&\le C + C \|f\|_{L^m(\Omega)} \left(\int_{\{u> s_2\}} u^{(\eta-\gamma_2)\frac{m}{m-1}}\right)^{\frac{m-1}{m}}.	
			\end{aligned}
		\end{equation}
Let observe that $\eta>\gamma_2$ implies $\eta + (p-1)(1-\theta)>0$ thanks to \eqref{condent}. Hence, for the left-hand side of \eqref{regsomm}, one can reason as follows
		\begin{equation}\label{regsomm2}
			\begin{aligned}
				\int_\Omega \left((1 + u)^{\frac{\eta + (p-1)(1-\theta)}{p}}-1\right)^{p^*} &\ge \int_{\{u > s_2\}} \left((1 + u)^{\frac{\eta + (p-1)(1-\theta)}{p}}-1\right)^{p^*}
				\\
				&\ge c_{t,s_2}\int_{\{u > s_2\}}  u^{\frac{(\eta + (p-1)(1-\theta))p^*}{p}},
			\end{aligned}	
		\end{equation}
where in the last inequality we use that $(1+x)^{t}-1 \ge c_{t,s_2}x^t$, for all $x >s_2$, $t>0$ and where $c_{t,s_2}$ depends only on $t$ and $s_2$. Now, inserting \eqref{regsomm2} in \eqref{regsomm}, one gets that
		\begin{equation}\label{lemsommleb}
			\begin{aligned}
				&\left(\frac{c_{t,s_2}^{\frac{1}{p^*}}}{\mathcal{S}}\right)^p \left(\int_{\{u> s_2\}}  u^{\frac{(\eta + (p-1)(1-\theta))p^*}{p}} \right)^{\frac{p}{p^*}}
				\\
				&\le C + C \|f\|_{L^m(\Omega)} \left(\int_{\{u> s_2\}} u^{(\eta-\gamma_2)\frac{m}{m-1}}\right)^{\frac{m-1}{m}}.
			\end{aligned}
		\end{equation}
		Let us fix
		\begin{equation}\label{eta}
			\eta = \frac{m\gamma_2(N-p) + N(m-1)(p-1)(1-\theta)}{N-mp},
		\end{equation}	
then
$$
{\frac{(\eta + (p-1)(1-\theta))p^*}{p}} = (\eta-\gamma_2)\frac{m}{m-1} = \frac{Nm((p -1)(1-\theta)+\gamma_2)}{N-mp}.
$$
This choice of $\eta$ gives that $\eta>\gamma_2$ thanks to the fact that $\theta<1+\frac{\gamma_2}{p-1}$. Now, since $m<\frac{N}{p}$, then $\frac{p}{p^*}>\frac{m-1}{m}$ and, thus, we deduce that $u^{\frac{Nm((p -1)(1-\theta)+\gamma_2)}{N-mp}}\in L^1(\Omega)$ after an application of the Young inequality.
		
\medskip

{\sl Proof of $(i)$.} Here $m \ge \max\left(\frac{p^*}{p^*-\theta(p-1)-1+\gamma_2},1\right)$ and we prove that $u$ has finite energy. We note that, if $\max\left(\frac{p^*}{p^*-\theta(p-1)-1+\gamma_2},1\right)=1$, then $\gamma_2\geq 1$ and $\theta\leq \frac{\gamma_2-1}{p-1}$, therefore the result follows from Lemma \ref{lem_sum_m=1}. Hence we assume $\frac{p^*}{p^*-\theta(p-1)-1+\gamma_2}>1$. Let us apply  \eqref{d14} with $\eta = \theta(p-1) +1$, yielding to
		\begin{equation}\label{lemregsob}
			\begin{aligned}
				& \int_\Omega |\nabla T_\sigma(u)|^p \le C + C\int_{\{u >s_2\}}   f u^{\theta(p-1) +1-\gamma_2},
			\end{aligned}
		\end{equation}
		where $C$ does not depend on $\sigma$. Moreover, since $\nabla u$ is almost everywhere finite in $\Omega$ (cf. Lemma \ref{lem_sum_m=1}), applying the Fatou Lemma to \eqref{lemregsob} as $\sigma\to \infty$, we deduce
		\begin{equation}\label{lemregsob2}
			\begin{aligned}
				& \int_\Omega |\nabla u|^p \le C + C\int_{\{u > s_2\}}   f u^{\theta(p-1) +1-\gamma_2}.
			\end{aligned}
		\end{equation}
It follows by an application of the H\"older inequality, which is admissible since $m \ge \frac{p^*}{p^*-\theta(p-1)-1+\gamma_2}>1$, and through the Sobolev inequality that
		\begin{equation}\label{lemregsob3}
			\begin{aligned}
				&\int_{\{u > s_2\}}   f u^{\theta(p-1) +1-\gamma_2} \le C\|f\|_{L^{m}(\Omega)} \left(\int_{\Omega} u^{p^*}\right)^{\frac{\theta(p-1) +1-\gamma_2}{p^*}}
				\\
				&\le C\mathcal{S}^{\theta(p-1) +1-\gamma_2}\|f\|_{L^{m}(\Omega)} \left(\int_{\Omega} |\nabla u|^p\right)^{\frac{\theta(p-1) +1-\gamma_2}{p}}.
			\end{aligned}
		\end{equation}
		
		Now, recalling that $\theta < 1 + \frac{\gamma_2}{p-1}$, one can apply the Young inequality to the right-hand side of \eqref{lemregsob3} and, combining with \eqref{lemregsob2}, one gets that $u$ belongs to $W^{1,p}_0(\Omega)$. \\
	
\medskip
		
{\sl Proof of $(ii)$.} Let $1<m<\frac{p^*}{p^*-\theta(p-1)-1+\gamma_2}$. We note that this interval is not empty if and only if $\theta>\frac{\gamma_2-1}{p-1}$. Let $\eta$  be as in \eqref{eta}. By assumption on $m$, we have $\theta(p-1) -\eta +1>0$. Moreover, reasoning as for obtaining \eqref{lemsommleb}, one has that $fu^{\eta-\gamma_2}\chi_{\{u>s_2\}}\in L^1(\Omega)$. Hence one can apply the Fatou Lemma to \eqref{d14} as $\sigma\to \infty$, obtaining that
		\begin{equation}\label{lemsommsob}
			\int_{\Omega} \frac{|\nabla u|^p}{(1+u)^{\theta(p-1) -\eta +1}}\le C + C\int_{\{u> s_2\}} f u^{\eta-\gamma_2}< \infty.
		\end{equation}
Now let $0<q<p$. Then it follows from the H\"older inequality that
		\begin{equation}\label{lemsommsob2}
			\begin{aligned}
				\int_{\Omega} |\nabla u|^q &= \int_{\Omega} \frac{|\nabla
					u|^q}{(1+u)^{\frac{q(\theta(p-1) -\eta +1)}{p}}}(1 + u)^{\frac{q(\theta(p-1) -\eta +1)}{p}}
				\\&\le \left(\int_{\Omega} \frac{|\nabla u|^p}{(1+u)^{\theta(p-1) -\eta +1}}\right)^\frac{q}{p} \left( \int_{\Omega} (1 + u)^{\frac{q(\theta(p-1) -\eta +1)}{p-q}}\right)^{\frac{p-q}{p}}
				\\
				& \stackrel{\eqref{lemsommsob}}{\le} C\left( \int_{\Omega} (1+u)^{\frac{q(\theta(p-1) -\eta +1)}{p-q}}\right)^{\frac{p-q}{p}}.
			\end{aligned}
		\end{equation}
We note that, since $\theta<1+\frac{\gamma_2}{p-1}$, we have $\frac{p^*}{p^*-\theta(p-1)-1+\gamma_2}<\frac{N}{p}$ and then we have already proven that $u^{\frac{Nm((p -1)(1-\theta)+\gamma_2)}{N-mp}}\in L^1(\Omega)$. Hence one can fix  $q=\frac{Nm((p-1)(1-\theta)+\gamma_2)}{N-m(\theta(p-1)+1-\gamma_2)}$ in order to have that the right-hand side of \eqref{lemsommsob2} is finite. This choice of $q$ is admissible, i.e. $0<q<p$, since $\theta<1+\frac{\gamma_2}{p-1}$ and $m<\frac{p^*}{p^*-\theta(p-1)-1+\gamma_2}$. This concludes the proof.
\end{proof}

Now we show that, as observed above, in the limit case $\theta=1+\frac{\gamma_2}{p-1}$, if $f$ is regular enough we have again regularity properties for $u$.
\begin{lemma}
\label{crit^2}
Let $a$ satisfy \eqref{cara1} and \eqref{b} with $\theta=1+\frac{\gamma_2}{p-1}$. Let $h$ satisfy \eqref{h2} and let $f\in L^\frac{N}{p}(\Omega)$ be nonnegative. Then any entropy solution $u$ to \eqref{pbmain} belongs to $W^{1,p}_0(\Omega) \cap L^q(\Omega)$ for every $q<\infty$.
\end{lemma}

\begin{proof}
Let $\eta>0$ and let us consider
$$
P(u)=((1+T_\sigma(u))^\eta-(1+k)^\eta)^+(1+T_\sigma(u))^{\gamma_2},
$$
where $\sigma>k>s_2$. We note that $\supp(P(u))=\{x\in \Omega : u>k\}$, and that, since $u$ is an entropy solution, thanks to Definition \ref{defent}, $P(u)\in \sob \cap L^\infty(\Omega)$. Moreover, using \eqref{cara1} and after straightforward computations, we obtain $a(x,u,\nabla P(u))\cdot \nabla P(u) \geq 0$ in $\Omega$.
Hence, applying Lemma \ref{P-test} and using \eqref{cara1}, we have
\begin{equation}
\label{ccc1}
\eta \int_{\{u>k\}} a(x,u,\nabla u)\cdot \nabla T_\sigma(u) (1+T_\sigma(u))^{\eta-1+\gamma_2} \leq \int_{\{u>k\}} h(u)fP(u).
\end{equation}
As regard the left-hand side of \eqref{ccc1}, since by assumption $\theta=1+\frac{\gamma_2}{p-1}$, it follows that
\begin{equation}
\begin{aligned}
\label{ccc2}
&\int_{\{u>k\}} a(x,u,\nabla u)\cdot \nabla T_\sigma(u) (1+T_\sigma(u))^{\eta-1+\gamma_2}
\\
&\stackrel{\eqref{cara1},\eqref{b}}\geq \alpha \int_{\{u>k\}} |\nabla T_\sigma(u)|^p (1+T_\sigma(u))^{\eta-p}
\\
& \stackrel{\phantom{\eqref{cara1},\eqref{b}}}= \frac{\alpha p^p}{\eta^p}\int_{\{u>k\}} |\nabla ((1+T_\sigma(u))^{\frac{\eta}{p}})|^p.
\end{aligned}
\end{equation}
Now we focus on the right-hand side of \eqref{ccc1}, obtaining
\begin{equation}
\begin{aligned}
\label{ccc3}
\int_{\{u>k\}} h(u)fP(u) &\stackrel{k>s_2,\eqref{h2}}\leq  c_2\int_{\{u>k\}} \left(1+\frac{1}{u}\right)^{\gamma_2}f((1+T_\sigma(u))^\eta-(1+k)^\eta)
\\
&\leq  c_2(1+s_2^{-1})^{\gamma_2} \int_{\{u>k\}} f((1+T_\sigma(u))^\eta-(1+k)^\eta)
 \\
& \leq  2^p c_2(1+s_2^{-1})^{\gamma_2} \int_{\{u>k\}} f((1+T_\sigma(u))^\frac{\eta}{p}-(1+k)^{\frac{\eta}{p}})^p \\
&\phantom{\leq} + 2^p c_2(1+s_2^{-1})^{\gamma_2}(1+k)^\eta \int_{\{u>k\}} f,
\end{aligned}
\end{equation}
where in the last step we used the inequality
\begin{equation}
\label{mah}
x^s-\tilde k^s \leq 2^s((x-\tilde k)^s+\tilde k^s)
\end{equation}
with $s=p$, $x=(1+T_\sigma(u))^\frac{\eta}{p}$ and $\tilde k=(1+k)^\frac{\eta}{p}$ (note that it is an admissible application since the inequality holds for every $x\geq \tilde k\geq 0$ and for $s>0$). Hence, using first the H\"older inequality and then the Sobolev inequality in the first summand of the last term on the right-hand side of \eqref{ccc3}, we deduce
\begin{equation}
	\begin{aligned}
\label{ccc4}
&\int_{\{u>k\}} h(u)fP(u)
\\
&\leq  2^p c_2(1+s_2^{-1})^{\gamma_2} \left(\int_{\{u>k\}} f^{\frac{N}{p}}\right)^{\frac{p}{N}}\left(\int_{\{u>k\}} ((1+T_\sigma(u))^\frac{\eta}{p}-(1+k)^{\frac{\eta}{p}})^{p^*}\right)^\frac{p}{p^*}  \\
& \phantom{\leq}+ Ck^\eta\|f\|_{L^1(\Omega)}  \\
&\leq  (2\mathcal S)^p c_2(1+s_2^{-1})^{\gamma_2} \left(\int_{\{u>k\}} f^{\frac{N}{p}}\right)^{\frac{p}{N}}\int_{\{u>k\}} |\nabla ((1+T_\sigma(u))^{\frac{\eta}{p}}-(1+k)^{\frac{\eta}{p}})|^p  \\
&\phantom{\leq} +  Ck^\eta\|f\|_{L^1(\Omega)}.
	\end{aligned}
\end{equation}
Thus, inserting \eqref{ccc2} and \eqref{ccc4} in \eqref{ccc1}, after some manipulations, we have
\begin{equation*}
\left(\alpha \frac{p^p}{\eta^p} - (2\mathcal S)^p c_2(1+s_2^{-1})^{\gamma_2} \left(\int_{\{u>k\}} f^{\frac{N}{p}}\right)^{\frac{p}{N}} \right)\int_{\{u>k\}} |\nabla ((1+T_\sigma(u))^{\frac{\eta}{p}})|^p \leq Ck^\eta\|f\|_{L^1(\Omega)}.
\end{equation*}
Since $u$ is almost everywhere finite in $\Omega$, thanks to the absolute continuity of the integral there exists $k_0>s_2$ such that
$$
(2\mathcal S)^p c_2(1+s_2^{-1})^{\gamma_2} \left(\int_{\{u>k_0\}} f^{\frac{N}{p}}\right)^{\frac{p}{N}}\leq \frac{\alpha p^p}{2\eta^p},
$$
which implies that for every $\sigma>k_0$ we have
\begin{equation}
\begin{aligned}
\label{ccc5}
\mathcal S^{-p}\left(\int_{\{u>k_0\}} ((1+T_\sigma(u))^\frac{\eta}{p}-(1+k_0)^{\frac{\eta}{p}})^{p^*}\right)^\frac{p}{p^*} &\leq \int_{\{u>k_0\}} |\nabla ((1+T_\sigma(u))^{\frac{\eta}{p}})|^p
\\
&\leq \frac{2C\eta^p k_0^\eta}{\alpha p^p}\|f\|_{L^1(\Omega)}.
\end{aligned}
\end{equation}
Using again that $u$ is almost everywhere finite in $\Omega$ and that $T_\sigma(u)\in \sob$ for every $\sigma>0$, by Lemma \ref{dalmaso2} we deduce that $\nabla T_\sigma(u)$ converges to $\nabla u$ almost everywhere in $\Omega$ as $\sigma\to \infty$. This implies, after an application of the Fatou Lemma to \eqref{ccc5}, that
\begin{equation}
	\begin{aligned}
\label{ccc6}
\mathcal S^{-p}\left(\int_{\{u>k_0\}} ((1+u)^\frac{\eta}{p}-(1+k_0)^{\frac{\eta}{p}})^{p^*}\right)^\frac{p}{p^*}
&\leq\int_{\{u>k_0\}} |\nabla ((1+u)^{\frac{\eta}{p}})|^p
\\
&\leq \frac{2C\eta^p k_0^\eta}{\alpha p^p}\|f\|_{L^1(\Omega)}.
\end{aligned}
\end{equation}
Now, choosing $\eta=p$ in \eqref{ccc6} and recalling that, since $u$ is an entropy solution, $T_{k_0}(u)\in \sob$, we have $u\in \sob$.
In what follows we prove that $u\in L^q(\Omega)$ for every $1\leq q<\infty$. We note that it is sufficient to prove that $\displaystyle \int_{\{u>k_0\}} u^q< \infty$. It follows from \eqref{mah} with $s=p^*$, $x=(1+u)^{\frac{\eta}{p}}$ and $\tilde k=(1+k_0)^{\frac{\eta}{p}}$ that
\begin{equation}
	\begin{aligned}
\label{ccc7}
\int_{\{u>k_0\}} (1+u)^\frac{N\eta}{N-p}
&\leq 2^{p^*}\left(\int_{\{u>k_0\}} ((1+u)^\frac{\eta}{p}-(1+k_0)^{\frac{\eta}{p}})^{p^*}\right)
\\
&+(2^{p^*}+1)(1+k_0)^\frac{N\eta}{N-p}|\Omega|.
\end{aligned}
\end{equation}
Thanks to \eqref{ccc6}, the right-hand side of \eqref{ccc7} is finite for any $\eta>0$, whence the result follows.
\end{proof}

\subsection{Bounded solutions}

Here we show that, if $f$ is sufficiently regular and \eqref{condent} holds, then any entropy solution to \eqref{pbmain} is bounded.

\begin{lemma}
Let $a$ satisfy \eqref{cara1}, \eqref{b} and \eqref{condent}. Let $h$ satisfy \eqref{h2} and let $f\in L^m(\Omega)$ be nonnegative with $m>\frac{N}{p}$. Then any entropy solution $u$ to \eqref{pbmain} belongs to $L^\infty(\Omega)$.
\end{lemma}
\begin{proof}
	
	We define $H:\R\to \R$ as
	$$
	\displaystyle H(s):=\int_0^s \frac{1}{(1+|t|)^{\theta-\frac{\gamma_2}{p-1}}} dt,
	$$
	noting that $H(s)\to \infty$ as $s\to \infty$ if and only if $\theta\leq 1+\frac{\gamma_2}{p-1}$ (i.e. \eqref{condent}). Hence, the proof of the lemma will be concluded once that it is shown that $H(u)$ is bounded. Let us consider
	$$
	P(u)=G_k(H(T_\sigma(u)))(1+T_\sigma(u))^{\gamma_2} \qquad\forall \sigma>k>H(s_2),
	$$
where $G_k$ is defined by \eqref{defGk}, and let observe that, since $u$ is an entropy solution and by definitions of $H$ and $G_k$, it belongs to $\sob\cap L^\infty(\Omega)$. Moreover one has that $\supp (P(u))= A_{k,\sigma}$ where
	$$
	A_{k,\sigma}:=\{x\in\Omega : H(T_{\sigma}(u(x)))>k\}.
	$$
	The previous choice of $P$ allows to apply Lemma \ref{P-test}, obtaining \eqref{P-formu}. As regards the left-hand side of \eqref{P-formu} we have, using the positivity of the principal operator and the definition of $H$, that
	\begin{equation}\label{stamp2}
		\begin{aligned}
			\int_{\Omega} a(x,u,\nabla u)\cdot \nabla P(u) &\stackrel{\eqref{cara1}}\geq \int_{A_{k,\sigma}} b(u)(1+u)^{\gamma_2}H'(u)|\nabla T_\sigma(u)|^p  \\
			&\stackrel{\eqref{b}}\geq \alpha \int_{A_{k,\sigma}} \frac{|\nabla T_\sigma(u)|^p}{(1+u)^{p\left(\theta-\frac{\gamma_2}{p-1}\right)}}  \\
			& =  \alpha\int_{A_{k,\sigma}} |\nabla G_k(H(T_\sigma(u)))|^p  \\
			& \geq  \frac{\alpha}{\mathcal S^p}\left(\int_{A_{k,\sigma}} G_k(H(T_\sigma(u)))^{p^*}\right)^{\frac{p}{p^*}},
		\end{aligned}
	\end{equation}
	where the last step follows by the Sobolev inequality. Now we focus on the right-hand side of \eqref{P-formu}; since $\sigma>k>H(s_2)$, we use \eqref{h2} obtaining
	\begin{equation}
		\begin{aligned}
			\label{stamp3}
			\int_{\Omega} h(u)f P(u) & \leq  c_2 \int_{\Omega} \frac{fG_k(H(T_\sigma(u)))(1+u)^{\gamma_2}}{u^{\gamma_2}}  \\
			&\stackrel{\eqref{as3}}\leq  c_2\overline c\int_{A_{k,\sigma}} fG_k(H(T_\sigma(u)))  \\
			&\leq  c_2\overline c \left(\int_{A_{k,\sigma}} f^{\frac{p^*}{p^*-1}}\right)^{\frac{p^*-1}{p^*}}\left(\int_{A_{k,\sigma}} G_k(H(T_\sigma(u)))^{p^*}\right)^{\frac{1}{p^*}},
		\end{aligned}
	\end{equation}
	where in the last inequality we applied the H\"older inequality. Gathering \eqref{stamp2} and \eqref{stamp3} in \eqref{P-formu}, we deduce that there exists $C>0$ such that
	\begin{eqnarray}
		\label{stamp4}
		\int_{A_{k,\sigma}} G_k(H(T_\sigma(u)))^{p^*} \leq C\left(\int_{A_{k,\sigma}} f^{\frac{p^*}{p^*-1}}\right)^{\frac{p^*-1}{p-1}} \stackrel{A_{k,\sigma} \subset A_k}\leq C\left(\int_{A_{k}} f^{\frac{p^*}{p^*-1}}\right)^{\frac{p^*-1}{p-1}}
	\end{eqnarray}
	where
	$$
	A_k:=\{x\in \Omega : H(u(x))>k \}.
	$$
	Now, by continuity of $G_k$ and $H$, we have $G_k(H(T_\sigma(u)))\chi_{A_{k,\sigma}}\to G_k(H(u))\chi_{A_k}$ almost everywhere in $\Omega$ as $\sigma\to \infty$. Hence, applying the Fatou Lemma to the left-hand side of \eqref{stamp4}, we obtain, for $l>k>H(s_2)$, that
	\begin{equation}
	\begin{aligned}
		\label{stamp5}
		(l-k)^{p^*}|A_l| & \leq  \int_{A_l} G_k(H(u))^{p^*} \leq \int_{A_k} G_k(H(u))^{p^*}  \\
		& \leq  C \|f\|_{L^m(\Omega)}^{\frac{p^*}{p-1}} |A_k|^{\frac{m(p^*-1)-p^*}{m(p-1)}},
	\end{aligned}
	\end{equation}
	in which we also applied the H\"older inequality to the right-hand side of \eqref{stamp4} with exponents $\frac{m(p^*-1)}{p^*}$ and $\frac{m(p^*-1)}{m(p^*-1)-p^*}$; let also note that, since $p<N$ and $m>\frac{N}{p}$, one has $\frac{m(p^*-1)}{p^*}>1$ and $\frac{m(p^*-1)-p^*}{m(p-1)}>1$. Finally, applying Lemma 4.1 of \cite{St}, it follows from \eqref{stamp5} that $H(u)\in L^{\infty}(\Omega)$, which implies the result.
\end{proof}

\section{Uniqueness of entropy solutions}
\label{s5}

In order to show the uniqueness result, we first show that an entropy solution satisfies a suitable truncated weak formulation (see also Section \ref{s7}).
\begin{lemma}
\label{enttoren}
Let $a$ satisfy \eqref{cara1} and \eqref{cara2}. Let $h$ satisfy \eqref{h2} and let $f\in L^1(\Omega)$ be nonnegative. Let $u$ be an entropy solution to \eqref{pbmain} then it holds
\begin{equation}
\label{quren}
\int_{\Omega} a(x,u,\nabla u) \cdot \nabla\varphi S(u) + \int_{\Omega} a(x,u,\nabla u) \cdot \nabla u S'(u)\varphi = \int_{\Omega} h(u)fS(u)\varphi,
\end{equation}
for every $S\in W^{1,\infty}(\R)$ with compact support and for every $\varphi\in \sob\cap L^{\infty}(\Omega)$.
\end{lemma}

\begin{proof}
Without loss of generality we suppose that $\supp(S)=[-M,M]$ for some $M>0$ so that $S(u)=S(T_M(u))$. Moreover let us underline that, since $T_M(u)\in \sob$ and $S\in W^{1,\infty}(\R)$, then $S(T_M(u))\in W^{1,p}(\Omega) \cap \linf$. It follows that $S(u)\varphi=S(T_M(u))\varphi\in \sob \cap \linf$ for every $\varphi\in \sob\cap L^{\infty}(\Omega)$. Hence we can choose $T_l(u)-S(u)\varphi$ as test function in \eqref{ent} with $k=\|S(u)\varphi\|_{\linf}+1$ and $l>\max(M,1)$ obtaining
\begin{equation}
\label{rin1}
\begin{aligned}
\int_{\Omega} a(x,u,\nabla u)\cdot\nabla T_k(u-T_l(u)+S(u)\varphi) \le \int_{\Omega}h(u)fT_k(u-T_l(u)+S(u)\varphi).
\end{aligned}
\end{equation}
Recalling that $\nabla(u-T_l(u))=\nabla u\chi_{\{u>l\}}$ and that $S(u)\varphi=0$ in $\{u>M\}$, we deduce
\begin{equation}
\label{rin2}
\begin{aligned}
&\int_{\Omega} a(x,u,\nabla u)\cdot\nabla T_k(u-T_l(u)+S(u)\varphi)
\\
&\stackrel{\eqref{cara1}}\geq  \int_{\{|u-T_l(u)+S(u)\varphi|\leq k\}}\!\!\!\!\!\!\!\!\!\!\!\!\!\! a(x,T_M(u),\nabla T_M(u))\cdot \nabla (S(T_M(u))\varphi).
\end{aligned}
\end{equation}
Now, using again that $T_M(u)\in\sob$, we have $|a(x,T_M(u),\nabla T_M(u))|\stackrel{\eqref{cara2}}\leq \beta(\ell(x) + T_M(u)^{p-1}+|\nabla T_M(u)|^{p-1}) \in L^{\frac{p}{p-1}}(\Omega)$ and $|\nabla (S(T_M(u))\varphi)|\in L^p(\Omega)$. Hence $a(x,T_M(u),\nabla T_M(u))\cdot \nabla (S(T_M(u))\varphi)\in L^1(\Omega)$. Moreover, since $\chi_{\{|u-T_l(u)+S(u)\varphi|\leq k\}}\to 1$ almost everywhere in $\Omega$ as $l\to \infty$ since $k=\|S(u)\varphi\|_{\linf}+1$ then an application of the Lebesgue Theorem gives that
\begin{equation}
\label{rin3}
\begin{aligned}
&\lim_{l\to \infty} \int_{\{|u-T_l(u)+S(u)\varphi|\leq k\}}\!\!\!\!\!\!\!\!\!\!\!\!\!\! a(x,T_M(u),\nabla T_M(u))\cdot \nabla (S(T_M(u))\varphi)
\\
&= \int_{\Omega} a(x,T_M(u),\nabla T_M(u))\cdot \nabla (S(T_M(u))\varphi).
\end{aligned}
\end{equation}
It follows from \eqref{rin1}, \eqref{rin2} and \eqref{rin3} that
\begin{equation}
\label{rin1bis}
\int_{\Omega} a(x,u,\nabla u)\cdot \nabla (S(u)\varphi) \leq \limsup_{l\to \infty}\int_{\Omega}h(u)fT_k(u-T_l(u)+S(u)\varphi).
\end{equation}
For the right-hand side of \eqref{rin1bis}, we have (recall that $l>M$)
\begin{eqnarray*}
\int_{\Omega}h(u)fT_k(u-T_l(u)+S(u)\varphi) &=& \int_{\{u\leq M\}} h(u)fS(u)\varphi  \nonumber \\
&& + \int_{\{u>l\}}h(u)fT_k(u-T_l(u)).
\end{eqnarray*}
Since $|h(u)fT_k(u-T_l(u))|\stackrel{\eqref{h2}}\leq k \displaystyle \left(\sup_{s\in (1,\infty)} h(s)\right) f\in L^1(\Omega)$ and since  $h(u)fT_k(u-T_l(u))\chi_{\{u>l\}}\to 0$ as $l\to \infty$, one can apply the Lebesgue Theorem, obtaining
\begin{equation}
\label{rin5}
\lim_{l \to \infty} \int_{\Omega}h(u)fT_k(u-T_l(u)+S(u)\varphi) = \int_{\{u\leq M\}} h(u)fS(u)\varphi = \int_{\Omega} h(u)fS(u)\varphi.
\end{equation}
Finally, combining \eqref{rin1bis} and \eqref{rin5}, we have shown that
$$
\int_{\Omega} a(x,u,\nabla u) \cdot \nabla\varphi S(u) + \int_{\Omega} a(x,u,\nabla u) \cdot \nabla u S'(u)\varphi \leq \int_{\Omega} h(u)fS(u)\varphi,
$$
for every $\varphi\in \sob\cap\linf$. The result follows repeating the same argument with $-\varphi$.
\end{proof}

The previous lemma allows us to prove the following result which is fundamental to show the uniqueness theorem.

\begin{lemma}\label{tecuni}
	Under the assumptions of Lemma \ref{enttoren}, let $a$ satisfy \eqref{cara3} and
	\begin{equation}\label{lipa}
		|a(x,s_1,\xi)- a(x,s_2,\xi)| \le C(\ell(x)+|\xi|^{p-1})|s_1-s_2|, \quad \forall s_1,s_2 \in \mathbb{R},
	\end{equation}
	for $C>0$, for every $\xi \in \mathbb{R}^N$ and for almost every $x$ in $\Omega$. Let $h$ be a decreasing function and let $f>0$ almost everywhere in $\Omega$. Let $u_1, u_2$ entropy solutions to \eqref{pbmain} such that:
	\begin{itemize}
		\item[$(i)$] $\| |a(x,T_{2n}(u_i),\nabla T_{2n}(u_i))| \|^{\frac{p}{p-1}}_{L^{\frac{p}{p-1}}(\Omega)} \le Cn$, for $i=1,2$ and $n\geq n_0$ for some $n_0>0$ where  $C$ does not depend on $n$;
		\item[$(ii)$] $\frac{1}{n}\||\nabla u_i| \chi_{\{n<u_i<2n\}}|\|^p_{L^p(\Omega)}$ converges to zero as $n\to\infty$ for $i=1,2$.
	\end{itemize}
	Then $u_1=u_2$ almost everywhere in $\Omega$.
\end{lemma}

\begin{proof}	
	Let $u_1$ and $u_2$ be two entropy solutions to \eqref{pbmain}. Then it follows from Lemma \ref{enttoren} that $u_1$ and $u_2$ satisfy \eqref{quren}. Let us now fix $S=V_n$ (see \eqref{Vdelta}) in both formulations \eqref{quren} and take $\varphi = \frac{T_k(u_1-u_2)}{k}V_n(u_2)$ and $\varphi = \frac{T_k(u_1-u_2)}{k}V_n(u_1)$ in the formulation \eqref{quren} solved, respectively, by $u_1$ and $u_2$. Let us remark that the previous functions are admissible test functions; indeed, for instance,  $\nabla (T_k(u_1-u_2)V_n(u_2))$ can be  different from zero only in the set $\{|u_1-u_2|\le k\}\cap \{u_2\le 2n\}$, then $\varphi$ belongs to $\sob\cap L^\infty(\Omega)$ since, by Definition \ref{defent}, $T_k(u_i)\in \sob$ for $i=1,2$ and for every $k>0$.
	\\
Subtracting the two formulations one gets
	\begin{equation}\label{uni1}
		\begin{aligned}
			&0= \frac{1}{k}\int_\Omega f (h(u_2)-h(u_1)) T_k(u_1-u_2)V_n(u_1)V_n(u_2)
			\\
			&+\frac{1}{k}\int_\Omega \left(a(x,u_1,\nabla u_1)-a(x,u_2,\nabla u_2)\right)\cdot \nabla T_k(u_1-u_2) V_n(u_1)V_n(u_2)
			\\
			&- \frac{1}{nk}\int_{\{n<u_2<2n\}} a(x,u_1,\nabla u_1) \cdot \nabla u_2 T_k(u_1-u_2)V_n(u_1)
			\\
			&+ \frac{1}{nk}\int_{\{n<u_1<2n\}} a(x,u_2,\nabla u_2) \cdot \nabla u_1 T_k(u_1-u_2)V_n(u_2)
			\\
			&- \frac{1}{nk}\int_{\{n<u_1<2n\}} a(x,u_1,\nabla u_1) \cdot \nabla u_1 T_k(u_1-u_2)V_n(u_2)
			\\
			&+ \frac{1}{nk}\int_{\{n<u_2<2n\}} a(x,u_2,\nabla u_2) \cdot \nabla u_2 T_k(u_1-u_2)V_n(u_1)
			\\
			& =: (A) + (B) + (C) + (D) + (E) + (F).
		\end{aligned}
	\end{equation}
Let us now take $k\to 0$ and $n\to \infty$ in the previous. We start from $(C)$; applying the H\"older inequality we have
$$
\begin{aligned}
	(C)\le \frac{1}{n} \left(\int_{\Omega} |a(x,T_{2n}(u_1),\nabla T_{2n}(u_1))|^{\frac{p}{p-1}}\right)^{\frac{p-1}{p}} \left(\int_{\{n<u_2<2n\}}|\nabla u_2|^p\right)^\frac{1}{p}
\end{aligned}
$$
and using assumptions $(i)$ and $(ii)$, the term on the right-hand side of the previous goes to zero as $n\to \infty$. This shows that both $(C)$ and $(D)$, $(E)$ and $(F)$ (with an analogous reasoning) go to zero uniformly in $k$ and as $n \to \infty$. \\
Now let us consider $(B)$. One has
$$
\begin{aligned}
	(B)= &\frac{1}{k}\int_\Omega (a(x,u_1,\nabla u_1) - a(x,u_1,\nabla u_2))\cdot \nabla T_k(u_1-u_2)V_n(u_1) V_n(u_2)
	\\
	& + \frac{1}{k}\int_\Omega (a(x,u_1,\nabla u_2) - a(x,u_2,\nabla u_2))\cdot \nabla T_k(u_1-u_2)V_n(u_1) V_n(u_2)=: (B_1) + (B_2).
\end{aligned}
$$
We can get rid of $(B_1)$ which is nonnegative by \eqref{cara3}. For $(B_2)$ one can use \eqref{lipa}, deducing that
$$
\begin{aligned}
(B_2)\le \int_{\{|u_1-u_2|<k\}} C(\ell(x) + |\nabla u_2|^{p-1})|\nabla (u_1-u_2)| V_n(u_1)V_n(u_2).
\end{aligned}
$$
Then one can simply pass the previous to the limit as $k\to 0$ in order to get that $(B_2)$ tends to zero since $\nabla T_{2n}(u_1) = \nabla T_{2n}(u_2)$ almost everywhere in $\{u_1=u_2\}$. \\
Finally, gathering all these calculations in \eqref{uni1} and using that $h$ is decreasing, we deduce
$$
0\le \frac{1}{k}\int_\Omega f (h(u_2)-h(u_1)) T_k(u_1-u_2)V_n(u_1)V_n(u_2) \le \omega_{k,n},
$$
where $\omega_{k,n}$ is a quantity which goes to zero as $k\to 0$ and $n\to\infty$. Therefore two applications of the Fatou Lemma, first as $k\to0$ and then as $n\to\infty$, gives that
$$\int_\Omega f |(h(u_2)-h(u_1))|=0.$$
Since $f>0$ almost everywhere in $\Omega$ then the previous yields to $u_1=u_2$ almost everywhere in $\Omega$.
\end{proof}

We are ready to prove Theorem \ref{teo_uniqueness}, whose proof is an application of Lemma \ref{tecuni}.

\begin{proof}[Proof of Theorem \ref{teo_uniqueness}]
Recall that $a(x,s,\xi)=\tilde a(x,\xi)b(s)$. We note that, under the assumption \eqref{caraunique} on $\tilde a$ and since $b$ is a bounded Lipschitz function, $a$ satisfies \eqref{cara1}, \eqref{cara2}, \eqref{cara3} and \eqref{lipa}. Thus, in order to have uniqueness of entropy solutions, we need to show that any entropy solution satisfies assumptions $(i)$ and $(ii)$ of Lemma \ref{tecuni}. \\
Let $u$ be an entropy solution then, by taking $\varphi = 0$ and $k=2n$ (with $n>1$) in \eqref{ent}, we have
\begin{equation}
\begin{aligned}\label{stimaunicita}
\int_\Omega b(T_{2n}(u))|\nabla T_{2n}(u)|^p &\stackrel{\eqref{caraunique}}\le  \int_{\Omega} b(T_{2n}(u)) \tilde a(x,\nabla T_{2n}(u))\cdot \nabla T_{2n}(u)  \\
& \leq  \int_{\{u\le 1\}} h(u)fu + 2n\int_{\{u>1\}} h(u)f \leq  Cn,
\end{aligned}
\end{equation}
where in the last inequality we use that, by \eqref{ent0}, $h(u)fu\chi_{\{u\le 1\}} \in L^1(\Omega)$ and, by \eqref{h2}, $h(u)f\chi_{\{u>1\}}\in L^1(\Omega)$.
It follows that
\begin{equation*}
\begin{aligned}
\int_\Omega |a(x, T_{2n}(u),\nabla T_{2n}(u))|^{\frac{p}{p-1}} &\stackrel{\eqref{caraunique}}\le \beta^{\frac{p}{p-1}}\int_{\Omega} b(T_{2n}(u))^{\frac{p}{p-1}}\left(\ell(x)+|\nabla T_{2n}(u)|^{p-1}\right)^{\frac{p}{p-1}}  \\
& \leq (2\beta)^{\frac{p}{p-1}} \int_\Omega b(T_{2n}(u))^{\frac{p}{p-1}}\left(\ell(x)^\frac{p}{p-1}+|\nabla T_{2n}(u)|^p\right)
\\
&\le Cn,
\end{aligned}
\end{equation*}	
where in the last inequality we use \eqref{stimaunicita} and that by assumptions $\ell\in L^{\frac{p}{p-1}}(\Omega)$ and $b$ is bounded from above. Therefore $(i)$ of Lemma \ref{tecuni} is proved. \\
In order to show $(ii)$ we take $\varphi = T_n(u)$ and $k=n$ in \eqref{ent}, then using \eqref{caraunique} and \eqref{b}, one obtains that
  $$
  \begin{aligned}
  \alpha\int_{\{n<u<2n\}}\frac{|\nabla u|^p}{(1+u)^{\theta(p-1)}} \le \int_{\{u>n\}}h(u)f T_n(u-T_n(u)).	
  \end{aligned}
  $$
Therefore from the previous and for $n$ large enough we deduce
    $$
  \begin{aligned}
  	\frac{\alpha}{(1+2n)^{\theta(p-1)}}\int_{\{n<u<2n\}}|\nabla u|^p \stackrel{\eqref{h2}}\le c_2n\int_{\{u>n\}}fu^{-\gamma_2}, 	
  \end{aligned}
  $$
  which implies
\begin{equation}\label{unique}
  \begin{aligned}
  	\frac{1}{n}\int_{\{n<u<2n\}}|\nabla u|^p \stackrel{\eqref{as3}}\le C\int_{\{u>n\}}fu^{\theta(p-1)-\gamma_2}. 	
  \end{aligned}
\end{equation}
If $\theta(p-1)-\gamma_2\le 0$ then it is sufficient requiring $f\in L^1(\Omega)$ in order to deduce that the right-hand side of \eqref{unique} tends to zero as $n\to\infty$. Otherwise, if $\theta(p-1)-\gamma_2> 0$ and $f\in L^m(\Omega)$ (with $m>1$), by an application of the H\"older inequality in \eqref{unique} we obtain
  \begin{equation}\label{unique2}
  	\begin{aligned}
  		\frac{1}{n}\int_{\{n<u<2n\}}|\nabla u|^p \le C\left(\int_{\{u>n\}}f^m\right)^{\frac{1}{m}} \left(\int_\Omega u^{\frac{m(\theta(p-1)-\gamma_2)}{m-1}}\right)^{\frac{m-1}{m}}, 	
  	\end{aligned}
  	\end{equation}
which gives also in this case that the right-hand side of \eqref{unique2} goes to zero as $n\to\infty$ since, by assumption, $u^{\frac{m(\theta(p-1)-\gamma_2)}{m-1}}\in L^1(\Omega)$. Hence the first two statements of the theorem follow.

\smallskip

Now we assume \eqref{condent} and \eqref{condmuni}. It follows from the first part of the proof and by \eqref{condmuni} that we just have to prove the case $\theta(p-1)-\gamma_2> 0$ in which
\begin{equation}
\label{appmcoroun}
m\geq \frac{N(p-1)}{(N-p)(\gamma_2-\theta(p-1))+N(p-1)}.
\end{equation}
Thus, if $\theta=1+\frac{\gamma_2}{p-1}$, we have $m\geq \frac{N}{p}$ and, thanks to Lemma \ref{crit^2}, $u\in L^q(\Omega)$ for every $q<\infty$. On the other hand, if $\theta<1+\frac{\gamma_2}{p-1}$, by Lemma \ref{lem_sum_m>1}, we have that $u^{\frac{Nm((p -1)(1-\theta)+\gamma_2)}{N-mp}} \in L^1(\Omega)$. Using \eqref{appmcoroun} we deduce that
$$
\frac{m(\theta(p-1)-\gamma_2)}{m-1} \leq \frac{Nm((p -1)(1-\theta)+\gamma_2)}{N-mp},
$$
hence $u^{\frac{m(\theta(p-1)-\gamma_2)}{m-1}} \in L^1(\Omega)$. Thus the result follows from the first part.
\end{proof}

\section{Distributional and renormalized solutions}
\label{s7}

In this section we prove, under a suitable restriction on $\theta$, that an entropy solution is actually a distributional solution. Indeed, as we have already seen above, $a(x,u,\nabla u)$ is not always locally integrable. This implies that the distributional notion of solution loses sense, but, as long as $u$ is such that $a(x,u,\nabla u)\in L^1_{\rm loc}(\Omega)$, one can show that an entropy solution to \eqref{pbmain} is also a distributional one (see Definition \ref{defdis} below). \\
Finally we introduce the notion of renormalized solution (see Definition \ref{renormalized} below) and we show the equivalence with the definition of entropy solution. 
\\

We first precise the meaning of distributional solution and then we state and prove the above mentioned result.

\begin{defin}
	\label{defdis}
	A nonnegative measurable function $u$ is a \textit{distributional solution} to problem \eqref{pbmain} if $T_k(u) \in W^{1,1}_0(\Omega)$ for every $k>0$ and if
	\begin{gather}
		a(x,u,\nabla u) \in L^{1}_{\rm loc}(\Omega)^N,\quad h(u)f\in L^1_{\rm loc}(\Omega), \nonumber \\
		\int_{\Omega} a(x,u,\nabla u)\cdot\nabla \varphi = \int_{\Omega}h(u)f \varphi \quad \forall \varphi \in C^1_c(\Omega).  \label{disw}
	\end{gather}
\end{defin}

Thus we have:

\begin{lemma}
\label{entdisgeneral}
Let $a$ satisfy \eqref{cara1}, $h$ satisfy \eqref{h2} and let $f\in L^1(\Omega)$ be nonnegative. If $u$ is an entropy solution to \eqref{pbmain} such that $a(x,u,\nabla u) \in L^1_{\rm loc}(\Omega)^N$, then $u$ is also a distributional solution to \eqref{pbmain}.
\end{lemma}

\begin{proof}
	First we note that, since $u$ is an entropy solution, it holds $T_k(u)\in \sob$ for every $k>0$. \\
	Now fix $\psi\in C_c^1(\Omega)$. Choosing $\varphi=T_l(u)-\psi\in \sob\cap \linf$ in \eqref{ent} with $l>1$ and $k=\|\psi\|_{\linf}+1$, we obtain
	\begin{equation}
		\label{d1}
		\int_{\Omega} a(x,u,\nabla u)\cdot\nabla T_k(u-T_l(u)+\psi) \le \int_{\Omega}h(u)fT_k(u-T_l(u)+\psi).
	\end{equation}
	We focus on the left-hand side of \eqref{d1}. We have
	\begin{eqnarray}
		\label{d6}
		\int_{\Omega} a(x,u,\nabla u)\cdot\nabla T_k(u-T_l(u)+\psi) & = & \int_{\{|u-T_l(u)+\psi|<k\}\cap \{u>l\}} a(x,u,\nabla u)\cdot \nabla u\nonumber \\
		&& +\int_{\{|u-T_l(u)+\psi|<k\}} a(x,u,\nabla u)\cdot \nabla \psi \nonumber \\
		&\stackrel{\eqref{cara1}}\geq & \int_{\{|u-T_l(u)+\psi|<k\}} a(x,u,\nabla u)\cdot \nabla \psi.
	\end{eqnarray}
	It follows from \eqref{d1} and \eqref{d6} that
	\begin{equation}
		\label{d7}
		\int_{\{|u-T_l(u)+\psi|<k\}} a(x,u,\nabla u)\cdot \nabla \psi \leq \int_{\Omega}h(u)fT_k(u-T_l(u)+\psi).
	\end{equation}
	Moreover, $|a(x,u,\nabla u)\cdot \nabla \psi\chi_{\{|u-T_l(u)+\psi|<k\}}|\leq \|\nabla \psi\|_{\linf} |a(x,u,\nabla u)|\in L^1(\supp \psi)$ and $\chi_{\{|u-T_l(u)+\psi|<k\}}\to 1$ a.e. in $\Omega$ (recall that $k=\|\psi\|_{\linf}+1$ and $u$ is almost everywhere finite on $\Omega$), hence, by the Lebesgue Theorem, we obtain
	\begin{equation}
		\label{d8}
		\lim_{l\to \infty} \int_{\{|u-T_l(u)+\psi|<k\}} a(x,u,\nabla u)\cdot \nabla\psi =\int_{\Omega} a(x,u,\nabla u)\cdot \nabla \psi.
	\end{equation}
	As regards the right-hand side of \eqref{d1}, we have
	\begin{eqnarray}
		\label{d2}
		\int_{\Omega}h(u)fT_k(u-T_l(u)+\psi) &=& \int_{\{u\leq l\}} h(u)f\psi^+ - \int_{\{u\leq l\}} h(u)f\psi^- \nonumber \\
		&&+ \int_{\{u> l\}} h(u)fT_k(u-T_l(u)+\psi).
	\end{eqnarray}
	Since $0\leq h(u)f\psi^{\pm}\chi_{\{u\leq l\}}$ is increasing in $l$, applying the Beppo Levi Theorem to the first two addends of the right-hand side of \eqref{d2}, we find
	\begin{equation}
		\label{d3}
		\lim_{l\to \infty} \int_{\{u\leq l\}} h(u)f\psi =\int_{\Omega}  h(u)f\psi.
	\end{equation}
	Moreover, $|h(u)fT_k(u-T_l(u)+\psi)\chi_{\{u>l\}}|\leq k\|h\|_{L^{\infty}((1,\infty))}f\in L^1(\Omega)$ and, since $u$ is almost everywhere finite on $\Omega$, $h(u)fT_k(u-T_l(u)+\psi)\chi_{\{u>l\}}\to 0$ a.e. in $\Omega$. It follows from Lebesgue Theorem that
	\begin{equation}
		\label{d4}
		\lim_{l\to \infty} \int_{\{u> l\}} h(u)fT_k(u-T_l(u)+\psi)=0.
	\end{equation}
	Combining \eqref{d3} and \eqref{d4} in \eqref{d2} we deduce
	\begin{equation}
		\label{d5}
		\lim_{l\to \infty} \int_{\Omega}h(u)fT_k(u-T_l(u)+\psi)= \int_{\Omega}  h(u)f\psi.
	\end{equation}
	Now letting $l$ tend to infinity in \eqref{d7}, thanks to \eqref{d8} and \eqref{d5}, we conclude that
	$$
	\int_{\Omega} a(x,u,\nabla u)\cdot \nabla \psi \leq \int_{\Omega}  h(u)f\psi \qquad \forall \psi \in C^1_c(\Omega).
	$$
	Repeating the same argument with $-\psi$ we find the inverse inequality, hence \eqref{disw} holds. \\
	Finally, using again that by assumption $a(x,u,\nabla u)\in L^1_{\rm loc}(\Omega)^N$, it follows from \eqref{disw} that $h(u)f \in L^1_{\rm loc}(\Omega)$.
\end{proof}

Note that, thanks to the regularity results of Section \ref{sec_reg}, we can give a summability range for $f$ in which an entropy solution $u$ satisfies $a(x,u,\nabla u)\in L^1_{\rm loc}(\Omega)^N$. Therefore, as an application of Lemma \ref{entdisgeneral}, $u$ is a distributional solution to \eqref{pbmain}. This is contained in the next lemma.

\begin{lemma}
		\label{entdis}
Let $a$ satisfy \eqref{cara1}, \eqref{b} and \eqref{condent} and let $h$ satisfy \eqref{h2}. Let $f\in L^m(\Omega)$ with
\begin{equation}
\label{conddis}
\begin{cases}
m\geq \max\left(\frac{N(p-1)}{(N(1-\theta)+1+\theta(p-1))(p-1)+\gamma_2(N-p+1)},1\right) & \text{if }\theta\neq\frac{1}{N-p+1}+\frac{\gamma_2}{p-1} \\
m>1 & \text{if }\theta=\frac{1}{N-p+1}+\frac{\gamma_2}{p-1}.
\end{cases}
\end{equation}
Then any entropy solution to \eqref{pbmain} it is also a distributional solution.
\end{lemma}

Note that
$$
\max\left(\tfrac{N(p-1)}{(N(1-\theta)+1+\theta(p-1))(p-1)+\gamma_2(N-p+1)},1\right)=1 \quad\text{if and only if}\quad 0\leq \theta\leq \frac{1}{N-p+1}+\frac{\gamma_2}{p-1}.
$$
Thus the exponents of summability for $f$ given in \eqref{conddis}  are continuous in $\theta= \frac{1}{N-p+1}+\frac{\gamma_2}{p-1}$. The only difference is that for $\theta=\frac{1}{N-p+1}+\frac{\gamma_2}{p-1}$ we need a bit more integrability than $f\in L^1(\Omega)$. Indeed, this value of $\theta$ is exactly the value for which we have $|\nabla u|\in M^{p-1}(\Omega)$ with $f\in L^1(\Omega)$, as observed in Remark \ref{rem1}. We also refer to Section \ref{radial} below for the optimality of the previous result.

\begin{proof}[Proof of Lemma \ref{entdis}]
Let $u$ be an entropy solution to \eqref{pbmain}. We divide in four cases depending on the range of $\theta$. \\
If $0\leq \theta <\frac{1}{N-p+1}+\frac{\gamma_2}{p-1}$, by \eqref{conddis}, we have $m\geq 1$, then it follows from Lemma \ref{lem_sum_m=1} and Remark \ref{rem1} that $u^{p-1}$ and $|\nabla u|^{p-1}$ belong to $L^1(\Omega)$. \\
If $\theta=\frac{1}{N-p+1}+\frac{\gamma_2}{p-1}$, by \eqref{conddis}, we have $m>1$. On the other hand, if $\frac{1}{N-p+1}+\frac{\gamma_2}{p-1}<\theta<1+\frac{\gamma_2}{p-1}$, again by \eqref{conddis}, we have $m\geq \tfrac{N(p-1)}{(N(1-\theta)+1+\theta(p-1))(p-1)+\gamma_2(N-p+1)}$. Then, in both cases, by Lemma \ref{lem_sum_m>1}, since $ \frac{Nm((p-1)(1-\theta)+\gamma_2)}{N-m(\theta(p-1)+1-\gamma_2)}\geq p-1$, we deduce again that $u^{p-1}$ and $|\nabla u|^{p-1}$ belong to $L^1(\Omega)$. \\
Finally, if $\theta=1+\frac{\gamma_2}{p-1}$, by \eqref{conddis}, we have $m\geq \frac{N(p-1)}{(N(1-\theta)+1+\theta(p-1))(p-1)+\gamma_2(N-p+1)}=\frac{N}{p}$ and so we can apply Lemma \ref{crit^2} obtaining $u\in \sob$. \\
In any cases, we have
\begin{equation*}
|a(x,u,\nabla u)|\stackrel{\eqref{cara2}}\leq \beta(\ell(x) + u^{p-1}+|\nabla u|^{p-1})\in L^1(\Omega).
\end{equation*}
Therefore, applying Lemma \ref{entdisgeneral}, the result follows.
\end{proof}

We conclude this section by specifying what we mean by a renormalized solution to \eqref{pbmain} and we show that this notion turns out to be equivalent with the notion of entropy solution whenever $f\in L^1(\Omega)$.

\begin{defin}
\label{renormalized}
A nonnegative measurable function $u$, which is almost everywhere finite in $\Omega$, is a \textit{renormalized solution} to problem \eqref{pbmain} if $T_k(u) \in W^{1,p}_0(\Omega)$ for every $k>0$ and if $a(x,T_k(u),\nabla T_k(u)) \in L^{\frac{p}{p-1}}(\Omega)^N$, $h(u)f S(u)\varphi \in L^1(\Omega)$,
\begin{equation}
\label{ren1}
\int_{\Omega} a(x,u,\nabla u)\cdot\nabla \varphi S(u) + \int_{\Omega}a(x,u,\nabla u)\cdot\nabla u S'(u)\varphi = \int_{\Omega}h(u)fS(u)\varphi
\end{equation}
for every $S \in W^{1,\infty}(\mathbb{R})$ with compact support and for every $\varphi \in W^{1,p}_0(\Omega)\cap L^\infty(\Omega)$,
and
\begin{equation}
\label{ren2}
\lim_{k\to \infty} \frac{1}{k}\int_{\{k< u< 2k\}} a(x,u,\nabla u)\cdot \nabla u \psi = 0
\end{equation}
for every $\psi\in C_b(\Omega)$ (i.e. the space of bounded continuous functions in $\Omega$).
\end{defin}

\begin{lemma}
\label{entequivren}
Let $a$ satisfy \eqref{cara1} and \eqref{cara2}. Let $h$ satisfy \eqref{h2} and let $f\in L^1(\Omega)$ be nonnegative. Then Definition \ref{defent} is equivalent to Definition \ref{renormalized}.
\end{lemma}
Let us just underline that, when proving Lemma \ref{enttoren} in the uniqueness section, we are showing nothing more than one of the requests of the Definition \ref{renormalized}.
\begin{proof}[Proof of Lemma \ref{entequivren}]
Let us assume that $u$ is an entropy solution. Firstly observe that it follows from Lemma \ref{enttoren} that \eqref{ren1} holds. Hence in order to prove that $u$ is a renormalized solution it is sufficient to show \eqref{ren2}. Note that $\nabla u\chi_{\{k<u<2k\}}=\nabla T_k(u-T_k(u))$, then taking $\varphi=T_k(u)$ with $k>1$ in \eqref{ent} we obtain
\begin{eqnarray*}
\frac{1}{k}\int_{\{k<u<2k\}} a(x,u,\nabla u)\cdot \nabla u\psi &\leq& \|\psi\|_{L^\infty(\Omega)}\frac{1}{k}\int_{\Omega} a(x,u,\nabla u)\cdot \nabla T_k(u-T_k(u)) \\
&\leq &  \|\psi\|_{L^\infty(\Omega)}\frac{1}{k} \int_{\Omega} h(u)f T_k(u-T_k(u)) \\
& \stackrel{\eqref{h2}}\leq &  \|\psi\|_{L^\infty(\Omega)} \left(\sup_{s\in [1,+\infty)}h(s)\right) \int_{\{u>k\}}f,
\end{eqnarray*}
for every $\psi \in C_b(\Omega)$. Recalling that $f\in L^1(\Omega)$, \eqref{ren2} follows. \\
Now let us assume that $u$ is a renormalized solution. Taking $\varphi=T_k(u-\phi)$ with $\phi\in W^{1,p}_0(\Omega)\cap L^{\infty}(\Omega)$ and $S(u)=V_m(u)$ ($V_m$ is defined in \eqref{Vdelta}) with $m> M:= k+\|\phi\|_{L^\infty(\Omega)}$ in \eqref{ren1}, we have, recalling that $\nabla T_k(u-\phi)=0$ where $u>M$, that
\begin{eqnarray}
\label{renentapp}
&& \int_\Omega a(x,u,\nabla u)\cdot\nabla T_k(u-\phi) -\frac{1}{m}\int_{\{m<u<2m\}}a(x,u,\nabla u)\cdot \nabla u T_k(u-\phi) \nonumber \\
&&= \int_{\{u<M\}} h(u)fT_k(u-\phi) + \int_{\{u\geq M\}} h(u)fS(u)T_k(u-\phi).
\end{eqnarray}
Now, thanks to \eqref{ren2}, the second term in the left-hand side of \eqref{renentapp} tends to $0$ as $m$ goes to infinity. As for the second term in the right-hand side of \eqref{renentapp} we have that $h(u)fS(u)T_k(u-\phi)\chi_{\{u\geq M\}}$ converges to $h(u)fT_k(u-\phi)\chi_{\{u\geq M\}}$ a.e. in $\Omega$ and $|h(u)fS(u)T_k(u-\phi)\chi_{\{u\geq M\}}|\leq k\left(\sup_{s\in [M,+\infty)}h(s)\right)f \in L^1(\Omega)$. Hence letting $m$ go to infinity, after an application of the Lebesgue Theorem, it follows from \eqref{renentapp} that
\begin{equation}
\label{renentapp2}
\int_\Omega a(x,u,\nabla u)\cdot\nabla T_k(u-\phi) = \int_{\Omega} h(u)fT_k(u-\phi),
\end{equation}
for every $\phi\in W^{1,p}_0(\Omega)\cap L^{\infty}(\Omega)$, i.e. \eqref{ent} with the equality sign. Moreover, since both $h(u)fS(u)T_k(u-\phi)$ and $h(u)fS(u)T_k(u-\phi)\chi_{\{u\geq M\}}$ belong to $L^1(\Omega)$, we deduce that $h(u)fT_k(u-\phi)\chi_{\{u< M\}}\in L^1(\Omega)$, which together to $h(u)fT_k(u-\phi)\chi_{\{u\geq M\}}\in L^1(\Omega)$ implies $h(u)fT_k(u-\phi)\in L^1(\Omega)$. Hence $u$ is an entropy solution.
\end{proof}

\section{Concluding remarks}
\label{s8}

In this final section we give some concluding remarks concerning several aspects of problem \eqref{pbintro}.

\subsection{Existence for any $\theta$ if $h$ touches zero}
Here we investigate a situation in which one could expect bounded solutions for any $\theta$, even if $\theta>1+\frac{\gamma_2}{p-1}$.

\smallskip

Let us consider problem \eqref{pbmain} where $h$ touches zero in some point. Precisely, let $\overline{s}$ be such that  $h(\overline{s}) =0$; under this condition one can prove the existence of a solution $u$ to \eqref{pbmain} such that $u\le \overline{s}$ almost everywhere in $\Omega$. Let us state and briefly prove the just mentioned result.
\begin{theorem}
Let $a$ satisfy \eqref{cara1}, \eqref{cara2}, \eqref{cara3} and \eqref{b}. Let $h$ satisfy \eqref{h1} with $\gamma_1\le 1$ and assume that there exists $\overline{s}$ such that $h(\overline{s})=0$. Let $f\in L^1(\Omega)$ be nonnegative. Then there exists a distributional solution $u$ to \eqref{pbmain} such that $u\le \overline{s}$ almost everywhere in $\Omega$.
\end{theorem}

\begin{proof}
	We give only a sketch of the proof. Even in this case  we work by approximation and the only difference with respect to the proof of Theorem \ref{teo_ent} is the estimates obtained on $u_n$. Here we slightly modify the approximation scheme \eqref{pbapprox}, namely $h_n(s):= T_n(h(s))$ if $s\le \overline{s}$ and $h_n(s):= 0$ if $s>\overline{s}$.
	Clearly, once again the existence of a weak solution $u_n\in\sob \cap L^\infty(\Omega)$ to \eqref{pbapprox} follows by \cite{ll}.
	\\
	In order to obtain the $L^\infty$-bound for $u_n$, we take $(1+T_n(u_n))^{\theta(p-1)}G_{\overline{s}}(u_n)$ as a test function in \eqref{pbapprox}, yielding to
	\begin{equation}\label{hzero}
		\alpha\int_\Omega |\nabla G_{\overline{s}}(u_n)|^p \le \int_{\{u_n>\overline s\}} h_n(u_n)f_n (1+T_n(u_n))^{\theta(p-1)}G_{\overline{s}}(u_n) =0,
	\end{equation}
since $h_n(u_n)=0$ in ${\{u_n>\overline s\}}$. Hence it follows from \eqref{hzero} that $u_n\le \overline{s}$ almost everywhere in $\Omega$. Now if one takes $(1+T_n(u_n))^{\theta(p-1)}u_n$ one simply gets that $u_n$ is bounded in $W^{1,p}_0(\Omega)$ with respect to $n$. From now on we can reason as in the proof of Theorem \ref{teo_ent} in order to show the existence of a bounded entropy solution to \eqref{pbmain}. Finally, since $u\in\sob$, thanks to Lemma \ref{entdisgeneral}, $u$ is a distributional solution.
\end{proof}

\subsection{The noncoercive case with model operator}

Let us highlight that existence of solution to \eqref{pbmain} in case of the model noncoercive operator $\displaystyle a(x,s,\xi)=\frac{|\xi|^{p-2}\xi}{(1+s)^{\theta(p-1)}}$ can be treat with a simple change of variable. \\
Indeed, let us consider the problem
\begin{equation}
\label{model1}
\begin{cases}
\displaystyle -\operatorname{div}\left(\frac{|\nabla u|^{p-2}\nabla u}{(1+u)^{\theta(p-1)}}\right) = \frac{f}{u^\gamma}  & \text{ in }\Omega, \\
u \ge 0 & \text{ in } \Omega,\\
u=0 & \text{ on } \partial\Omega,
\end{cases}
\end{equation}
with $\gamma>0$, and the following change of variable
\begin{equation*}
v= \Phi(u):=\int_{0}^{u} \frac{1}{(1+t)^\theta}\ dt.
\end{equation*}
Then one formally yields to
\begin{equation}
\label{model2}
\begin{cases}
\displaystyle -\Delta_p v = \frac{f}{\Phi^{-1}(v)^\gamma} & \text{ in }\Omega, \\
v \ge 0 & \text{ in } \Omega,\\
v=0 & \text{ on } \partial\Omega.
\end{cases}
\end{equation}	
After some calculations one can deduce that $\Phi^{-1}(s)^\gamma \sim s^\gamma$ as $s\to 0^+$. This equivalence allows to study the noncoercive problem \eqref{model1} once that one knows existence, uniqueness and regularity of solutions to problem \eqref{model2} which has been widely studied. \\
An important remark is that solutions to \eqref{model2} with $p=2$ and $f$ positive and regular enough belong to $H^1_0(\Omega)$ if and only if $\gamma <3$ (Theorem 2 of \cite{lm}). This result can be generalized for $1<p<N$ proving that a solution belongs to $\sob$ if and only if $\gamma< 2+\frac{1}{p-1}$. Hence we deduce that, if $\gamma\geq 2+\frac{1}{p-1}$ we cannot have solutions of \eqref{model1} with finite energy for every $f$.

\subsection{An instance of radial solution}
\label{radial}

Let $B_1(0) \subset \mathbb{R}^N$ be the $N$-dimensional ball ($N\ge 3$) of radius $1$ centered at the origin and let us consider
\begin{equation}\label{pbradial}
\begin{cases}
\displaystyle -\operatorname{div}\left(\frac{\nabla u}{(1+u)^\theta}\right)=\frac{C}{|x|^{N-\varepsilon}(1+u)^{\gamma_2}} & \text{ in }B_1(0), \\
u\geq 0 & \text{ in }B_1(0), \\
u=0 & \text{ on }\partial B_1(0),
\end{cases}
\end{equation}
for $\varepsilon>0$ and a suitable $C >0$. Note that $f\in L^m(\Omega)$ for every $1\leq m<\frac{N}{N-\varepsilon}$ and for every $\varepsilon>0$. Then a solution can be given as $u(x)= |x|^\alpha-1$ with $$\alpha = \frac{2+\varepsilon - N}{1-\theta +\gamma_2}.$$
\\
We note that $\alpha$ is negative thanks to the request  assumed throughout the paper, given by $\theta\le 1 +\gamma_2$ ($p=2$); in particular if $\theta = 1 +\gamma_2$ then $\alpha\to-\infty$ and the solution formally loses any Marcinkiewicz regularity (in agreement with the results of Section \ref{sec_reg}). After straightforward computations we obtain that $u\in W^{1,1}_0(\Omega)$ only if $\theta<\frac{1+\varepsilon}{N-1} +\gamma_2$. This suggests that finding solutions which, in general, do not have an integrable gradient is not a technical limit and, thus, the notion of entropy solution is more suitable to this kind of problems.  Furthermore, the solution to \eqref{pbradial} is clearly an entropy one and the upper limit of $\theta$ for which we have integrable gradient tends to $\frac{1}{N-1}+\gamma_2$ as $\varepsilon\to 0$ while $f$ loses its integrability. This shows that the threshold which appears in Lemma \ref{entdis} is optimal.

\subsection{The strongly singular case}

The aim of this section is to briefly highlight the main features of problem \eqref{pbmain} if $h$ satisfies \eqref{h1} with $\gamma_1>1$, i.e. in presence of a strongly singular $h$. In this case, once again through the scheme of approximation \eqref{pbapprox}, coherently to the case $\theta=0$, one reaches to solutions with truncations having only local finite energy.
Therefore, this means that the natural setting to work in is no more the \textit{entropy} framework, but the \textit{distributional} one.
\\
Let us give an idea that, if $u$ is a solution to \eqref{pbmain}, then $T_k(u) \in W^{1,p}_{\rm loc}(\Omega)$ for any $k>0$. Indeed one can formally multiply \eqref{pbmain} (which can be made rigorous at an approximation level by giving an analogue of Lemma \ref{lem_convqo}) by $(T_k(u)-k)\varphi^p$ ($0\le \varphi \in C^1_c(\Omega)$) and after an application of the Young inequality one gets
\begin{equation*}
	\alpha \int_\Omega \frac{|\nabla T_k(u)|^p\varphi^p}{(1+T_k(u))^{\theta(p-1)}} + 	p\int_\Omega a(x,T_k(u),\nabla T_k(u))\cdot \nabla \varphi (T_k(u)-k)\varphi^{p-1}\le 0,
\end{equation*}
where we used that $(T_k(u)-k)\varphi^p\le 0$ if $u<k$ and that $(T_k(u)-k)\varphi^p= 0$ if $u\ge k$, and we also employed \eqref{cara1} and \eqref{b}.
\\
Now, recalling \eqref{cara2}, one can apply the Young inequality obtaining that
\begin{equation*}
	\int_\Omega |\nabla T_k(u)|^p\varphi^p \le C,
\end{equation*}
where $C$ clearly depends on $k$.
\\
The other important difference between the case $\gamma_1\le 1$ and $\gamma_1> 1$ is how to give meaning to the boundary datum. As we have seen, since $h(u)u$ is not necessarily finite in zero, we are not able to deduce that $u$ is zero in the usual sense of the Sobolev trace as for the case $\gamma_1\le1$. On the other hand, one can observe that $h(u)u^{\gamma_1}$ is finite in zero and this allows to show that $T_k^{\frac{\gamma_1-1+p}{p}}(u)$ belongs to $W^{1,p}_0(\Omega)$ for any $k>0$, which is the weaker way in which we mean the Dirichlet datum in this case.
\\
Let us also highlight that also Lemmas \ref{convgradienti} and \ref{stronconv} can be suitably adapted by localizing the proof.
\\
For the sake of completeness we state the existence result for this case. As already said, the aim is proving the existence of a distributional solution; therefore, coherently to the case $\gamma_1\le1$, one needs to require some regularity on $f$ in order to deduce that $a(x,u,\nabla u)$ is locally integrable (cf. Lemma \ref{entdis}).
\\
The existence theorem can be stated as follows.

\begin{theorem}
Let $a$ satisfy \eqref{cara1}, \eqref{cara2}, \eqref{cara3}, \eqref{b} and \eqref{condent}. Let $h$ satisfy \eqref{h1} and \eqref{h2}. Let $f\in L^m(\Omega)$ be nonnegative such that \eqref{conddis} holds.
Then there exists a measurable nonnegative function $u$ such that $T_k^{\frac{\gamma_1-1+p}{p}}(u)$ belongs to $W^{1,p}_0(\Omega)$ for any $k>0$ and
	\begin{gather}
		a(x,u,\nabla u) \in L^{1}_{\rm loc}(\Omega)^N,\quad h(u)f\in L^1_{\rm loc}(\Omega),  \nonumber\\
		\int_{\Omega} a(x,u,\nabla u)\cdot\nabla \varphi = \int_{\Omega}h(u)f \varphi \quad \forall \varphi \in C^1_c(\Omega).\nonumber
	\end{gather}
\end{theorem}

\subsection*{Acknowledgements} The first author gratefully thanks the Department of Mathematics ``Ulisse Dini'' of the University of Florence for having partially supported this work with a Research Grant. The authors have been partially supported by GNAMPA of INdAM (Project ``Analisi di fenomeni di wetting in presenza di potenziali singolari'') and by PON Ricerca e Innovazione 2014-2020.

\end{document}